\documentclass[a4paper, 10pt]{preprint}

\usepackage[left=1.5in, right=1.5in, bottom=1.5in]{geometry}
\usepackage{comment}
\usepackage{mhenvs}
\usepackage{mhsymb}
\usepackage{orcidlink}

\usepackage[full]{textcomp}
\usepackage[osf]{newtxtext}
\usepackage{amssymb}
\usepackage{amsmath}
\usepackage{mathtools}
\usepackage{mhequ}
\usepackage{booktabs}
\usepackage{tikz}
\usepackage{mathrsfs}
\usepackage{longtable}
\usepackage{microtype}
\usepackage{wasysym}
\usepackage{centernot}
\usepackage{enumitem}

\usepackage{hyperref}

\usepackage{crossreftools}
\usepackage{physics}
\usepackage[bbgreekl]{mathbbol}
\usepackage{pgfplots}

\pgfplotsset{width=6.7cm,compat=1.9}
\usepackage[normalem]{ulem}
\mathtoolsset{showonlyrefs=true}

\newif\ifdark
\IfFileExists{dark}{\darktrue}{\darkfalse}

\ifdark

\definecolor{darkred}{rgb}{0.9,0.2,0.2}
\definecolor{darkblue}{rgb}{0.7,0.3,1}
\definecolor{darkgreen}{rgb}{0.1,0.9,0.1}
\definecolor{pagebackground}{rgb}{.15,.21,.18}
\definecolor{pageforeground}{rgb}{.84,.84,.85}
\pagecolor{pagebackground}
\AtBeginDocument{\globalcolor{pageforeground}}

\else

\definecolor{darkred}{rgb}{0.7,0.1,0.1}
\definecolor{darkblue}{rgb}{0.4,0.1,0.8}
\definecolor{darkgreen}{rgb}{0.1,0.7,0.1}
\definecolor{pagebackground}{rgb}{1,1,1}
\definecolor{pageforeground}{rgb}{0,0,0}

\fi

\definecolor{newred}{RGB}{208,16,76}
\definecolor{newgreen}{RGB}{34,125,81}

\makeatletter
\newcommand{\globalcolor}[1]{%
	\color{#1}\global\let\default@color\current@color
}
\makeatother

\DeclareSymbolFont{timesoperators}{T1}{ptm}{m}{n}
\SetSymbolFont{timesoperators}{bold}{T1}{ptm}{b}{n}

\makeatletter
\renewcommand{\operator@font}{\mathgroup\symtimesoperators}
\makeatother

\DeclareMathAlphabet{\mathbbm}{U}{bbm}{m}{n}
\DeclareFontFamily{U}{BOONDOX-calo}{\skewchar\font=45 }
\DeclareFontShape{U}{BOONDOX-calo}{m}{n}{
	<-> s*[1.05] BOONDOX-r-calo}{}
\DeclareFontShape{U}{BOONDOX-calo}{b}{n}{
	<-> s*[1.05] BOONDOX-b-calo}{}
\DeclareMathAlphabet{\mcb}{U}{BOONDOX-calo}{m}{n}
\SetMathAlphabet{\mcb}{bold}{U}{BOONDOX-calo}{b}{n}

\makeatletter 

\newcommand*{\fat}{}% Check if undefined
\DeclareRobustCommand*{\fat}{%
	\mathbin{\mathpalette\bigcdot@{}}}
\newcommand*{\bigcdot@scalefactor}{.5}
\newcommand*{\bigcdot@widthfactor}{1.15}
\newcommand*{\bigcdot@}[2]{%
	\sbox0{$#1\vcenter{}$}% math axis
	\sbox2{$#1\cdot\m@th$}%
	\hbox to \bigcdot@widthfactor\wd2{%
		\hfil
		\raise\ht0\hbox{%
			\scalebox{\bigcdot@scalefactor}{%
				\lower\ht0\hbox{$#1\bullet\m@th$}%
			}%
		}%
		\hfil
	}%
}

\DeclareRobustCommand{\TitleEquation}[2]{\texorpdfstring{\StrLeft{\f@series}{1}[\@firstchar]$\if b\@firstchar\boldsymbol{#1}\else#1\fi$}{#2}}

\makeatother

\theoremnumbering{Alph}
\renewtheorem{theorem*}{Theorem}
\renewtheorem{corollary*}{Corollary}

\renewtheorem{example*}{Example}
\numberwithin{equation}{section}

\def\slash{\leavevmode\unskip\kern0.18em/\penalty\exhyphenpenalty\kern0.18em}
\def\dash{\leavevmode\unskip\kern0.18em--\penalty\exhyphenpenalty\kern0.18em}

\let\epsilon\varepsilon

\def\${|\!|\!|}

\setlist{noitemsep,topsep=4pt}
\def\para_#1{/\!\!/_{\!#1}}

\def\f{\frac}

\begin{document}

\title{Statistical Inference for Homogenization Limits Driven by Wiener or Hermite Processes}
\author{Pablo Ramses Alonso-Martin}
\date{\today}

\institute{University of Warwick, Dept. of Statistics}
\maketitle
\begin{abstract}
We study the effective estimation of the diffusivity and Hurst parameter for the homogenized limit of a class of slow/fast systems. Depending on the system parameters, this limit solves a stochastic differential equation driven by either a Wiener process or a Hermite process. In the class of models we consider, the fast variable is a fractional Ornstein--Uhlenbeck process. We show that estimators constructed from the homogenized limit remain consistent when applied to appropriately subsampled data generated by the original slow/fast system.

A key tool in our analysis is the consistency of renormalized quadratic variations for a family of additive functionals of the fast process. Using Wiener chaos expansions, we obtain an \(L^2\)-orthogonal decomposition of these renormalized quadratic variations. This allows us to show that, under appropriate subsampling conditions, the consistency properties of the estimators are preserved even when the data is generated by the slow/fast system rather than the homogenized limit. We also show that, under stricter subsampling conditions, a non-central limit theorem is preserved in the case where the fluctuations of the estimator around the true value are non-Gaussian. As a direct consequence of convergence in \(L^2\), we obtain consistency of an estimator for the limiting self-similarity that does not require knowledge of the limiting diffusivity. Finally, we show that our results apply to a class of one-dimensional fluctuation models.

\end{abstract}
\tableofcontents
\section{Introduction}
The primary goal of this article is to study the behaviour of the quantity
\begin{equation}
    \label{eq: estimators}
C_{\varepsilon,\delta} = \begin{cases}
    N^{2H^*(m)-1}\sum_{i=0}^{N-1} \bigl(X^{\varepsilon}_{(i+1)\delta}-X^{\varepsilon}_{i\delta}\bigr)^2 \quad &H^*(m)>1/2 \\
    \sum_{i=0}^{N-1} \bigl(X^{\varepsilon}_{(i+1)\delta}-X^{\varepsilon}_{i\delta}\bigr)^2 \quad &H^*(m)\leq 1/2
    \end{cases}
\end{equation}
where \(X^{\varepsilon}_t\) is the first component of the solution to the system of equations 
\begin{equation}
\label{eq: multiscale system}
\begin{cases}
dX^{\varepsilon}_t = \alpha(\varepsilon, H^*(m)) G(Y^{\varepsilon}_t) dt \\
dY^{\varepsilon}_t = -\f{1}{\varepsilon}Y_t^{\varepsilon} dt + \f{\sigma}{\varepsilon^H} dB^H_t
\end{cases}
\end{equation}
for some \(G\in L^2(\mathbb{R}, \gamma)\) with Hermite rank \(m\geq 1\). Here \(B^H_t\) is a fractional Brownian motion with Hurst parameter 
\(H\in (0,1)\) and \(\gamma = \mathcal{N}(0, 1)\) is the invariant measure of the fast process \(Y^{\varepsilon}_t\). The quantity \(H^*(m) = (H-1)m + 1\) 
determines the qualitative behaviour of the limiting process and the scaling factor \(\alpha(\varepsilon, H^*(m))\) is defined as
\begin{equation}
\label{eq: scaling factor}
\alpha(\varepsilon, H^*(m)) = \begin{cases}
    \varepsilon^{H^*(m)-1} \quad &H^*(m)>1/2 \\
    \f{1}{\sqrt{\varepsilon}|\log(\varepsilon)|} \quad &H^*(m)=1/2 \\
    \f{1}{\sqrt{\varepsilon}} \quad &H^*(m)<1/2
    \end{cases}
\end{equation}
As shown in \cite{campeseContinuousBreuerMajor2020,gehringerFunctionalLimitTheorems2022a} under suitable assumptions on \(G\) the slow variable \(X^{\varepsilon}_t\) converges weakly in appropriate path spaces to an effective process \(X_t\) which, depending on the regularity of the fast process and the Hermite rank of \(G\), is either a rescaled Wiener process or a rescaled Hermite process of order \(m\):
\begin{equation}
\label{eq: effective limit}
X_t \;=\; \begin{cases}
\sqrt{C}\,W_t, & H^*(m)\leq 1/2,\\[0.4em]
\sqrt{C}\,Z_t^{H^*(m),m}, & H^*(m)>1/2,
\end{cases}
\end{equation}
where the scaling constant \(C=\mathbb{E}[X_1^2]\) encodes the size of the fluctuations of the limit.

The interest in such quantities \eqref{eq: estimators} stems from their use as statistical estimators for the fluctuation size and Hurst parameter in the limiting equation of the slow process \(X^{\varepsilon}_t\) as \(\varepsilon\rightarrow 0\). 
Such systems arise naturally in the context of rough homogenization for random ODEs of the form 
\[\dot{x}^{\varepsilon}_t = h(x^{\varepsilon}_t)g(y^{\varepsilon}_t) + \alpha(\varepsilon)G(y^{\varepsilon}_t)\]
where the solution to the slow equation in \eqref{eq: multiscale system} acts as a driver but also plays the role of a second order term in the perturbation expansion of the vector field that determines the evolution of interest, 
allowing us to study the fluctuations of the slow process around its averaged limit.

The weak convergence stated above makes the pair \((H', C)\), with \(H' = H^*(m)\vee 1/2\) the self-similarity index of the limiting process, the natural target of statistical inference: together they determine the second-order fluctuation structure of \(X^{\varepsilon}\). A finer inferential question would be to characterize the limiting process uniquely, which in the regime \(H^*(m)>1/2\) requires also the Hermite rank \(m\) of \(G\). Estimating \(m\) directly would, however, demand either observing the fast process \(Y^{\varepsilon}\) or knowing \(G\) explicitly---both of which defeat one of the main motivations for multiscale inference---and we therefore leave it outside the scope of this work. For most applications it is in any case sufficient to distinguish a Wiener limit (\(H^*(m)\leq 1/2\)) from a long-range dependent one (\(H^*(m)>1/2\)), and, within the latter regime, the Gaussian case \(m=1\) (fractional Brownian motion) from the non-Gaussian cases \(m\geq 2\), the first of which is the Rosenblatt process \cite{veillettePropertiesNumericalEvaluation2013}. Beyond \(m=2\) little is known about the distribution of Hermite processes and, to the best of our knowledge, their simulation has only been addressed very recently \cite{ayacheNumericalSimulationGeneralized2025}, so resolving \(m\) past the Gaussian/non-Gaussian dichotomy is of limited practical interest.

We follow a \textbf{plug-in} inference strategy: we take estimators that are known to be consistent for the parameters of the effective equation and apply them to discrete observations of the multiscale process \(X^{\varepsilon}\), asking under what joint regimes of the sampling step \(\delta\) and the scale separation \(\varepsilon\) the classical guarantees survive. The quantity \(C_{\varepsilon,\delta}\) defined in \eqref{eq: estimators} is precisely the estimator we use for \(C\): when \(H^*(m)\leq 1/2\) the limit is a Wiener process and \(C_{\varepsilon,\delta}\) reduces to its standard quadratic variation; when \(H^*(m)>1/2\) the limit exhibits long-range dependence and consistency of \(C_{\varepsilon,\delta}\) on the effective process is known from \cite{tudorVariationsFractionalBrownian} for the fractional Brownian case \(m=1\), from \cite{tudorVariationsEstimatorsSelfsimilarity2009} for the Rosenblatt case \(m=2\), and from \cite{chronopoulouSelfsimilarityParameterEstimation2011} for higher orders. To establish consistency of \(C_{\varepsilon,\delta}\) on the multiscale process we follow a strategy similar to \cite{tudorVariationsEstimatorsSelfsimilarity2009,chronopoulouSelfsimilarityParameterEstimation2011}, with the additional complication of having to control the multiscale nature of the problem.

Because the normalization of \(C_{\varepsilon,\delta}\) depends explicitly on \(H'\), estimating \(C\) presupposes access to the self-similarity index, so we also need an estimator for \(H'\). We consider the ratio estimator
\begin{equation}
\label{eq: estimator of H}
\hat{H}_{\varepsilon,\delta} = \f{1}{2} - \f{1}{2\log 2}\log\!\left(
\f{\sum_{i=0}^{2N-1}\bigl(X^{\varepsilon}_{(i+1)\delta/2}-X^{\varepsilon}_{i\delta/2}\bigr)^2}
{\sum_{i=0}^{N-1}\bigl(X^{\varepsilon}_{(i+1)\delta}-X^{\varepsilon}_{i\delta}\bigr)^2}\right),
\end{equation}
which compares quadratic variations on two nested partitions. A more orthodox transformation of a single quadratic variation would require the knowledge of \(C\) to estimate \(H'\); the ratio is instead scale-invariant, and once \(L^2\)-consistency of \(C_{\varepsilon,\delta}\) is established, consistency of \(\hat{H}_{\varepsilon,\delta}\) follows by the continuous mapping theorem. This motivates the two-step procedure we analyze throughout the article: first estimate \(H'\) via \(\hat{H}_{\varepsilon,\delta}\), then plug the estimate into the normalization of \(C_{\varepsilon,\delta}\) to recover \(C\).

\subsection{Organization of the article and Notation}
The rest of the article is organized as follows. In Section 2 we introduce all the preliminary concepts required for the rest of the manuscript. In Section 3 we motivate the need for a subsampling scheme when observing the slow process of a multiscale system like \eqref{eq: multiscale system} by 
showing that the quadratic variation of the slow process is asymptotically biased when the process is observed at an unconditional high-frequency. In Section 4 we show that under an appropriate subsampling rate the renormalized quadratic variation is asymptotically unbiased and we
provide explicit rates. Finally, specializing to the case where the estimator belongs to a finite Wiener chaos we show in Section 5 that the renormalized quadratic variation is consistent in all cases and in Section 6 we prove a non-Central Limit Theorem that resembles the classical results
for Hermite processes but using multiscale observations.

We will use the notation \(A(x)\sim B(x)\) to denote \(\lim_{x}\f{A(x)}{B(x)} = 1\) where \(x\) will either go to \(\infty\) or to \(0\) depending on the context. In other words, it denotes asymptotic equivalence not just up to the order of magnitude but 
also to proportionality constants. We will also use the notation \(A(x)\asymp B(x)\) to denote asymptotic equivalence up to a constant. That means \(\lim_{x}\f{A(x)}{B(x)} = C<\infty\) for some unspecified constant \(|C|>0\).
We will also use \(A(x)\lesssim B(x)\) to denote that \(\lim \frac{A(x)}{B(x)} = 0\) where \(x\) may vanish or blow up depending on the context. This is equivalent to being asymptotically bounded up to an unspecified constant.

\section{Preliminaries}
\subsection{Analysis on the Wiener space}
We work on an irreducible Gaussian probability space \((\Omega,\mathcal{F},\mathbb{P})\) carrying an isonormal Gaussian family \(\{W(h):h\in H\}\) indexed by the separable Hilbert space \(H=L^2(\mathbb{R},dt)\); that is, \(h\mapsto W(h)\) is linear and \(\mathbb{E}[W(h)W(g)]=\langle h,g\rangle_H\). In our setting \(W(h)=\int_{\mathbb{R}}h(s)\,dW_s\) is the Wiener integral against a standard Brownian motion \(W_t\).

The Hermite polynomials
\[H_n(x)=(-1)^n e^{x^2/2}\f{d^n}{dx^n}e^{-x^2/2}\]
form an orthogonal basis of \(L^2(\mathbb{R},\gamma)\) with \(\gamma=\mathcal{N}(0,1)\): every \(G\in L^2(\mathbb{R},\gamma)\) admits an expansion \(G(x)=\sum_{n\geq m}c_n H_n(x)\), and the smallest \(m\) with \(c_m\neq 0\) is the \emph{Hermite rank} of \(G\). Their behaviour under jointly Gaussian arguments is characterised by the following classical identity, which we will invoke throughout.
\begin{proposition}
\label{prop: hermite orthogonality}
Let \((X,Y)\) be a jointly Gaussian vector with \(\mathbb{E}[X]=\mathbb{E}[Y]=0\) and \(V[X]=V[Y]=1\). Then
\[\mathbb{E}[H_m(X)H_n(Y)]=\begin{cases} n!\,\mathbb{E}[XY]^n, & m=n, \\ 0, & m\neq n.\end{cases}\]
\end{proposition}
Denote by \(\mathcal{H}_n\) the \(L^2(\Omega)\)-closure of \(\mathrm{Span}\{H_n(W(h)):\|h\|_H=1\}\). We then have the orthogonal decomposition \(L^2(\Omega,\mathcal{F},\mathbb{P})=\bigoplus_{n\geq 0}\mathcal{H}_n\).

For \(f\in L^2(\mathbb{R}^n)\) the multiple Wiener--It\^o integral
\[I_n(f)=\int_{\mathbb{R}^n}f(x_1,\dots,x_n)\,dW(x_1)\cdots dW(x_n)\]
is the unique linear isometry \(L^2(\mathbb{R}^n)\to L^2(\Omega)\) extending the natural definition on product indicators. Writing \(\mathrm{Sym}(f)\) for the symmetrisation of \(f\) one has \(I_n(f)=I_n(\mathrm{Sym}(f))\) and
\[\mathbb{E}[I_p(f)I_q(g)]=\delta_{pq}\,p!\,\langle \mathrm{Sym}(f),\mathrm{Sym}(g)\rangle_{L^2(\mathbb{R}^p)}.\]
For symmetric \(f\in L^2(\mathbb{R}^p)\), \(g\in L^2(\mathbb{R}^q)\) and \(0\leq r\leq p\wedge q\), the \(r\)-contraction
\begin{multline*}(f\otimes_r g)(x_1,\dots,x_{p-r},y_1,\dots,y_{q-r})= \\ \int_{\mathbb{R}^r}f(x_1,\dots,x_{p-r},s_1,\dots,s_r)\,g(y_1,\dots,y_{q-r},s_1,\dots,s_r)\,ds_1\cdots ds_r\end{multline*}
is an element of \(L^2(\mathbb{R}^{p+q-2r})\); we write \(f\tilde{\otimes}_r g:=\mathrm{Sym}(f\otimes_r g)\) for its symmetrisation and \(f^{\otimes n}(x_1,\dots,x_n)=f(x_1)\cdots f(x_n)\) for the \(n\)-th tensor power of \(f\). Multiple Wiener--It\^o integrals obey the product formula
\begin{equation}
\label{eq: product structure}
I_p(f)\,I_q(g)=\sum_{r=0}^{p\wedge q} r!\binom{p}{r}\binom{q}{r}\,I_{p+q-2r}(f\tilde{\otimes}_r g),
\end{equation}
which is the main computational tool in the chaos decomposition of \(C_{\varepsilon,\delta}\). Finally, \(I_n\) realises an isometric isomorphism between symmetric \(L^2(\mathbb{R}^n)\) and \(\mathcal{H}_n\) via
\[H_n(W(h))=I_n(h^{\otimes n}),\qquad \|h\|_H=1.\]

\subsection{Fractional Brownian Motion}
A fractional Brownian motion (fBM) with Hurst parameter \(H\in(0,1)\) is the centred continuous Gaussian process \((B^H_t)_{t\geq 0}\) with covariance
\[\mathbb{E}[B^H_t B^H_s]=\f{1}{2}(t^{2H}+s^{2H}-|t-s|^{2H}).\]
The Mandelbrot--Van Ness representation \cite{mandelbrotFractionalBrownianMotions1968} expresses it as a Wiener integral against a standard Brownian motion,
\[B^H_t=\tilde{c}_H^{-1}\int_{\mathbb{R}}K^H(t,s)\,dW_s,\]
with an explicit kernel \(K^H\) whose precise form we will not need. As a consequence, integration against \(B^H\) reduces to integration against \(W\): it is shown in \cite{pipirasIntegrationQuestionsRelated2000,pipirasAreClassesDeterministic2001} that there exists a linear operator \(T_H\), given by fractional integration when \(H>1/2\) and by fractional differentiation when \(H<1/2\), such that
\begin{equation}
\label{eq: fBM Wiener integral relation}
\int_{\mathbb{R}}f(s)\,dB^H_s=\int_{\mathbb{R}}T_H f(s)\,dW_s
\end{equation}
for every \(f\) with \(T_H f\in L^2(\mathbb{R})\), the identity holding in \(L^2(\Omega)\).

\subsection{Fast component: time-rescaled fractional Ornstein--Uhlenbeck process}
We take \(Y^{\varepsilon}_t\) to be the unique stationary solution to
\begin{equation}
\label{eq: fOU SDE}
dY_t=-\f{1}{\varepsilon}Y_t\,dt+\f{\sigma}{\varepsilon^H}dB^H_t,
\end{equation}
with \(\sigma^2=\f{2}{\Gamma(2H+1)}\) chosen so that \(\mathbb{E}[(Y^{\varepsilon}_t)^2]=1\). It admits the closed-form representation
\[Y^{\varepsilon}_t=\int_{\mathbb{R}}P_{t,\varepsilon}(s)\,dB^H_s,\qquad P_{t,\varepsilon}(s)=\f{\sigma}{\varepsilon^H}e^{-(t-s)/\varepsilon}\1_{(-\infty,t]}(s),\]
so \eqref{eq: fBM Wiener integral relation} identifies \(Y^{\varepsilon}_t\) as an element of the first Wiener chaos of \(W\):
\begin{equation}
\label{eq: fOU as Wiener integral}
Y^{\varepsilon}_t=p\int_{\mathbb{R}}K_{t,\varepsilon}(s)\,dW_s,\qquad K_{t,\varepsilon}:=T_H P_{t,\varepsilon}.
\end{equation}
Since \(P_{t,\varepsilon}(s)=P_{0,\varepsilon}(s-t)\) and \(T_H\) commutes with translations, the kernel \(K_{t,\varepsilon}\) is a time-shift of its value at \(t=0\): setting \(K^{\varepsilon}:=K_{0,\varepsilon}\) and writing \((\tau_t f)(u):=f(u-t)\) for the time-shift operator, we have \(K_{t,\varepsilon}=\tau_t K^{\varepsilon}\). This shorthand will be used throughout the chaos decomposition of \(C_{\varepsilon,\delta}\) as it will prove very useful when dealing with inner products. For \(H>1/2\), the kernel \(K^{\varepsilon}\) admits the explicit representation
\begin{equation}
\label{eq: fOU kernel representation}
K^{\varepsilon}(s) = \varepsilon^{-1/2}e^{-s/\varepsilon}\int_0^{s/\varepsilon}e^v\,v^{H-3/2}\,dv\qquad (s>0),
\end{equation}
and carries a normalisation constant \(p>0\) given by
\begin{equation}
\label{eq: p constant}
p^2=\f{1}{\beta(H-1/2,\,2H-2)\,\Gamma(2H-1)},
\end{equation}
which we record here because it will play an important role in the kernel computations of Section~\ref{sec: convergence inner product appendix}.

The covariance \(\rho(s):=\mathbb{E}[Y^1_0 Y^1_s]\), extended symmetrically by \(\rho(-s)=\rho(s)\), and its rescaled version \(\rho^{\varepsilon}(s):=\mathbb{E}[Y^{\varepsilon}_0 Y^{\varepsilon}_s]=\rho(s/\varepsilon)\), exhibit the long-range decay
\begin{equation}
    \label{eq: rho asymptotic}
\rho(s)\sim \sigma^2 H(2H-1)\,s^{2H-2}=:\kappa_H\,s^{2H-2}\qquad (s\to\infty),
\end{equation}
which follows from taking \(N=1\) in Theorem 2.3 of \cite{cheriditoFractionalOrnsteinUhlenbeckProcesses2003}. The leading constant \(\kappa_H\) is negative for \(H<1/2\) and equals \(1/\Gamma(2H-1)\) for \(H>1/2\); it appears as the leading prefactor of the chaos-asymptotic estimates of Section~\ref{sec: proofs of consistency}, where the relation \(\rho(s)^q\sim \kappa_H^q\,|s|^{q(2H-2)}\) for \(q\ge 1\) will be used repeatedly.

\subsection{Hermite Processes}
For \(H\in(1/2,1)\) and \(m\in\mathbb{N}\) the \emph{Hermite process of order \(m\) and self-similarity \(H\)} is the element of the \(m\)-th Wiener chaos of \(W\) defined by
\begin{equation}
\label{eq: gehringer hermite process}
Z_t^{H,m}=\frac{K(H,m)}{m!}\,I_m(f_t^{H,m}),\qquad f_t^{H,m}(x_1,\dots,x_m)=\int_0^t\prod_{j=1}^{m}(s-x_j)_{+}^{-\left(\f{1}{2}+\f{1-H}{m}\right)}\,ds,
\end{equation}
where the normalising constant is given explicitly by
\begin{equation}
\label{eq: hermite normalisation}
K(H,m)=\left(\frac{m!\,H(2H-1)}{\beta\!\left(\f{1}{2}+\f{H-1}{m},\,2\f{1-H}{m}\right)^{m}}\right)^{1/2},
\end{equation}
so that \(\mathbb{E}[(Z_1^{H,m})^2]=1\); see \cite{taqquConvergenceIntegratedProcesses1979,gehringerFunctionalLimitTheorems2022} for a derivation. The process \(Z^{H,m}\) is self-similar with parameter \(H\), has stationary increments, and covariance
\[\mathbb{E}[Z_t^{H,m}Z_s^{H,m}]=\f{1}{2}\bigl(t^{2H}+s^{2H}-|t-s|^{2H}\bigr).\]
For \(m=1\), \eqref{eq: gehringer hermite process} reduces to the Mandelbrot--Van Ness representation of a fBM; for \(m=2\) it is the Rosenblatt process, and for \(m\geq 2\) the process is non-Gaussian.

\subsection{Constants}
\label{subsec: constants}
Throughout the paper we use \(K\) for a \emph{generic} positive constant, depending only on \(H, m\) and the function \(G\), whose value is irrelevant and may change from line to line.

For each \(0\leq r\leq m-1\) we denote by \(D^{\mathrm{Diag}}_{H,m,r}\) and \(D^{\mathrm{Off}}_{H,m,r}\) the leading constants of \(E^{\mathrm{Diag}}_{2m-2r}\) and \(E^{\mathrm{Off}}_{2m-2r}\) respectively, computed explicitly in Lemmas~\ref{lem: asymptotic variance m=1 H>1/2}--\ref{lem: asymptotics m>1 H<1/2}. The single chaos-asymptotic prefactor we use by name elsewhere is the leading constant of the second chaos at \(r=m-1\) in the regime \(H^*(m)>1/2\):
\begin{equation}
    \label{eq: D star}
    D_{\star}(H,m):=\frac{4\,(1/\Gamma(2H-1))^{2m}}{(4H-3)(2H-1)\bigl((m-1)(2H-2)+2\bigr)^{2}\bigl((m-1)(2H-2)+1\bigr)^{2}},
\end{equation}
read off Lemma~\ref{lem: asymptotic variance m>1 H>1/2} (and Lemma~\ref{lem: asymptotic variance m=1 H>1/2} when \(m=1\)); it enters the proof of Theorem~\ref{thm: non central limit} (Rosenblatt convergence). All other prefactors in Sections~\ref{sec: consistency} and~\ref{sec: proofs of consistency} are kept generic.

Finally, the auxiliary integral
\begin{equation}
    \label{eq: a Hmr def}
    a(H,m,r):=\int_0^1 (1-x)\,x^{2(m-r)(2H-2)}\,dx,
\end{equation}
well-defined when \(2(m-r)(2H-2)>-1\), appears in the Riemann-sum asymptotics of the off-diagonal contributions.

\section{Sampling without scale separation}
\label{sec: failure modes}
The two results below highlight what goes wrong if the sampling scheme does not take the scale-separation parameter \(\varepsilon\) into account, and together motivate the rest of the article. Proposition~\ref{prop: C to 0} shows that if one samples on scales finer than \(\varepsilon\) --- that is, \(\delta/\varepsilon\to 0\) --- the quadratic-variation functional \(C_{\varepsilon,\delta}\) collapses to zero. Proposition~\ref{prop: Hhat to 1} shows that sampling arbitrarily fast at a fixed \(\varepsilon\) (the natural setting where one ignores the multiscale structure altogether) leads the Hurst-parameter estimator to \(1\), regardless of the true asymptotic Hurst exponent \(H^{*}(m)\). The common thread is that on scales smaller than \(\varepsilon\) the slow process is effectively smooth and carries no information about its limiting self-similarity exponent: any reasonable estimator must therefore probe it on scales larger than \(\varepsilon\).
\begin{proposition}\label{prop: C to 0}
Let $X^\epsilon$ be the solution of \eqref{eq: multiscale system} with \(0<H<1\), \(G\in L^2(\mathbb{R}, \gamma)\) with Hermite rank \(m\). Let $T=N \delta$ for any time horizon $T$. If $\delta(\varepsilon)>0$ is such that \(\delta(\varepsilon)\xrightarrow{\varepsilon\rightarrow 0} 0\) and \(\delta(\varepsilon)/\varepsilon\rightarrow 0\) then 
\[
\mathbb{E}\left[C_{\delta, \epsilon}\right] \xrightarrow{\epsilon \rightarrow 0^{+}} 0
\]
\end{proposition}
\begin{proof}
We can again assume \(T=1\) for simplicity. We first derive the crude estimate
\[
\mathbb{E}[(X^{\epsilon}_{\delta} - X^{\epsilon}_0)^2] = \alpha(H^*(m),\varepsilon)^2\int_0^{\delta}\int_0^{\delta}\mathbb{E}[G(Y^{\epsilon}_s)G(Y^{\epsilon}_u)]dsdu\leq \alpha(H^*(m),\varepsilon)^2\delta^2 \mathbb{E}[G(Y_0)^2]
\]
from which, if \(H^*(m)\neq 1/2\)
\[
0\leq \mathbb{E}[C_{\delta, \varepsilon}]\leq \alpha(H^*(m),\varepsilon)^2\delta^{(2-2H^*(m))\wedge 1} \mathbb{E}[G(Y_0)^2] = \left(\frac{\delta}{\varepsilon}\right)^{(2-2H^*(m))\wedge 1}\mathbb{E}[G(Y_0)^2]\]
and if \(H^*(m) = 1/2\) then \[\mathbb{E}[C_{\delta, \varepsilon}] = \frac{\delta}{\varepsilon|\log(\varepsilon)|}\mathbb{E}[G(Y_0)^2] \]
and the result follows. 
\end{proof}

\begin{proposition}\label{prop: Hhat to 1}
Let $X^\varepsilon$ be the solution of \eqref{eq: multiscale system} and fix $T=N \delta$ for any time horizon $T$, $0<H<1$ and \(G\in L^p(\mathbb{R}, \gamma)\) for some \(p>2\)
such that \(\{ x\in\mathbb{R} : G(x)= 0\}\) has measure zero under \(\gamma\). Then for any \(\varepsilon>0\) fixed 
\[
\lim _{\delta \rightarrow 0} H^*_{\varepsilon, \delta}=1 
\]
in probability.
\end{proposition}
\begin{proof}
We assume without loss of generality that \(T=1\). The idea here is to re-scale the quantity of interest and define 
\[A(\varepsilon, \delta) = \frac{\sum_{i=0}^{N-1}(X^{\varepsilon}_{(i+1)\delta}-X^{\varepsilon}_{i\delta})^2}{\alpha(H^*(m),\varepsilon)^2\delta}\]
to then show that it converges to the constant \(\mathbb{E}(G(Y_0)^2)\) in two steps. In step 1 we show that 
\[
\left|A(\varepsilon,\delta) - \delta\sum_{i=0}^{N-1}G(Y_{i\delta/\varepsilon})^2\right|\xrightarrow{\varepsilon\rightarrow 0} 0 \quad \text{ in }L^1(\Omega, \mathbb{P})
\]
and then in step 2 we prove that 
\[\left|\delta\sum_{i=0}^{N-1}G(Y_{i\delta/\varepsilon})^2 - \int_0^1 G(Y_{s/\varepsilon})^2 ds\right|\xrightarrow{\varepsilon\rightarrow 0} 0 \quad \text{ in }L^1(\Omega, \mathbb{P})\]
and in a step 3 we conclude the result using these. 

\textbf{Step 1}. Note that
\[
A(\varepsilon, \delta) = \frac{\sum_{i=0}^{N-1}(X^{\varepsilon}_{(i+1)\delta}-X^{\varepsilon}_{i\delta})^2}{\alpha(H^*(m),\varepsilon)^2\delta} = 
\delta^{-1}\sum_{i=0}^{N-1}\int_{i\delta}^{(i+1)\delta}\int_{i\delta}^{(i+1)\delta}G(Y^{\varepsilon}_t)G(Y^{\varepsilon}_s)dsdt
\]
therefore,
\begin{align*}
A(\varepsilon, \delta) - \delta\sum_{i=0}^{N-1}G(Y_{i\delta}^{\varepsilon})^2 =& \delta^{-1}\sum_{i=0}^{N-1}\int_{i\delta}^{(i+1)\delta}\int_{i\delta}^{(i+1)\delta}(G(Y^{\varepsilon}_u)G(Y^{\varepsilon}_v)- G(Y^{\varepsilon}_{i\delta})^2)dudv \\
=& \delta \sum_{i=0}^{N-1}\frac{\int_{i\delta}^{(i+1)\delta}\int_{i\delta}^{(i+1)\delta}(G(Y^{\varepsilon}_u)G(Y^{\varepsilon}_v)- G(Y^{\varepsilon}_{i\delta})^2)dudv}{(\delta)^2}.
\end{align*}
Taking modulus and expectation on both sides we deduce that, by stationarity
\[
\mathbb{E}\left[\left|A(\varepsilon, \delta)-\delta^{-1}\sum_{i=0}^{N-1}G(Y^{\varepsilon}_{i\delta})^2\right| \right] \leq \frac{\int_{0}^{\delta/\varepsilon}\int_{0}^{\delta/\varepsilon}\mathbb{E}[|G(Y_u)G(Y_v) - G(Y_0)^2|]dudv}{(\delta/\varepsilon)^2}
\]
after applying Fubini on the right-hand side and using the fact that \(Y_t\) is stationary and so is \(G(Y_t)\) when \(G\) is measurable. 

If \(F(u,v) = \mathbb{E}[|G(Y_u)G(Y_v) - G(Y_0)^2|]\) is continuous at \((0,0)\) then we can conclude the result taking limits on both sides and using the fact that 
\[\lim_{t\rightarrow0}\frac{1}{t^2}\int_0^t\int_0^tF(u,v)dudv = F(0,0)\]

We thus only need to show that \[\lim_{(u,v)\rightarrow(0,0)}F(u,v) = 0\] for which we only need that \[G(Y_t)\xrightarrow{t\rightarrow 0}G(Y_0)\] in \(L^2(\Omega, \mathbb{P})\) since \begin{multline*}
\mathbb{E}[|G(Y_u)G(Y_v) - G(Y_0)^2|]\leq \\ ||G(Y_u)||_{L^2(\Omega)}||G(Y_v)-G(Y_0)||_{L^2(\Omega)}+||G(Y_0)||_{L^2(\Omega)}||G(Y_u)-G(Y_0)||_{L^2(\Omega)}.
\end{multline*} We first note that, since \(G\in L^2(\mathbb{R}, \gamma)\) then \(G(Y_t)\in L^2(\Omega, \mathbb{P})\) as \[\int_{\Omega}G(Y_t(\omega))^2d\mathbb{P} = \int_{\mathbb{R}}G(x)^2d\gamma < \infty.\]
We will proceed by a standard approximation argument, showing the result for a continuous function \(G\) and relying on continuous functions being dense on \(L^2(\mathbb{R}, \gamma)\). Let \(G\) be a continuous function, as \(Y_t\) has continuous paths almost surely we have pointwise convergence in \(\Omega\) in the sense of \[\lim_{t\rightarrow0}G(Y_t(\omega)) = G(Y_0(\omega))\] for almost all \(\omega\in\Omega\). Using stationarity and the integrability properties mentioned earlier it also holds that 
\[
\sup_{t\geq 0}\mathbb{E}[|G(Y_t)|^q] = \mathbb{E}[|G(Y_0)|^q]<\infty
\]
for all \(2< q \leq p\). This implies that the family \((G(Y_t)^2)_{t\geq 0}\) is uniformly integrable in \((\Omega, \mathcal{F}, \mathbb{P})\) and this immediately ensures convergence in \(L^2(\Omega, \mathbb{P})\) using Vitali's Theorem \cite{benedettoIntegrationModernAnalysis2009}. The argument to extend to not necessarily continuous functions \(G\) is very standard. 

\textbf{Step 2}. We now show that 
\begin{equation}
\delta\sum_{i=0}^{N-1}G(Y_{i\delta}^{\varepsilon})^2\xrightarrow[\delta\rightarrow 0]{L^1} \int_0^1 G(Y^{\varepsilon}_s)^2ds
\end{equation}
The result is almost straightforward for a continuous function \(G\) given that 
\[
\delta\sum_{i=0}^{N-1}G(Y^{\varepsilon}_{i\delta})^2 \xrightarrow[\delta\rightarrow 0]{\text{a.s.}} \int_0^1G(Y^{\varepsilon}_s)^2ds 
\]
via a Riemann sum argument. Now simply noting that 
\begin{equation}
\mathbb{E}\left[\left(\delta\sum_{i=0}^{N-1}G(Y^{\varepsilon}_{i\delta})^2\right)^2\right] = \delta^2\sum_{i,j}\mathbb{E}\left[G(Y^{\varepsilon}_{i\delta})^2G(Y^{\varepsilon}_{j\delta})^2\right]
\leq \mathbb{E}[(G(Y_0)^4)]<\infty
\end{equation}
we get that the sequence is Uniformly Integrable and thus concluding the convergence in \(L^1\). Now for a general \(G\in L^4\) we simply proceed again by density choosing a sequence of continuous functions \((G^n)_{n\in\mathbb{N}}\)
such that for any given \(\eta>0\) there exists \(n_{\eta}\) such that for any \(n\geq n_{\eta}\) it holds that \(||G_n^2 - G^2||_{L^1}\leq ||G_n^2 - G^2||_{L^2}<\eta\). Then 
\begin{multline}
\left|\delta\sum_{i=0}^{N-1}G(Y^{\varepsilon}_{i\delta})^2 - \int_0^1G(Y^{\varepsilon}_s)^2 ds\right|\leq \\
\delta\sum_{i=0}^{N-1}|G(Y_{i\delta}^{\varepsilon})^2 - G_n(Y_{i\delta}^{\varepsilon})^2|+|\delta\sum_{i=0}^{N-1}G_n(Y^{\varepsilon}_{i\delta})^2 - \int_0^1 G(Y^{\varepsilon}_s)^2ds|
\end{multline}
Taking now expectation on both sides and limits as \(\delta\rightarrow 0\) we get that 
\[
\lim_{\delta\rightarrow 0}\mathbb{E}\left|\delta\sum_{i=0}^{N-1}G(Y^{\varepsilon}_{i\delta})^2 - \int_0^1G(Y^{\varepsilon}_s)^2 ds\right| < \eta + \mathbb{E}\left[\left|\int_0^1 G_n(Y^{\varepsilon}_{s})^2 - G(Y^{\varepsilon}_s)^2ds\right|\right] < 2\eta
\]
from which the result follows as \(\eta\) is arbitrarily small. 

\textbf{Step 3}. The ratio inside the logarithm in \eqref{eq: estimator of H} may be written as 
\begin{align}
\frac{\sum_{i=0}^{2N-1}\left(X^{\varepsilon}_{(i+1)\delta/2}-X^{\varepsilon}_{i\delta/2}\right)^2}{\sum_{i=0}^{N-1}\left(X^{\varepsilon}_{(i+1)\delta}-X^{\varepsilon}_{i\delta}\right)^2} = \frac{\frac{\sum_{i=0}^{2N-1}\left(X^{\varepsilon}_{(i+1)\delta/2}-X^{\varepsilon}_{i\delta/2}\right)^2}{\alpha(H^*(m), \varepsilon)^2 \delta/2}\alpha(H^*(m), \varepsilon)^2 \delta/2}{\frac{\sum_{i=0}^{N-1}\left(X^{\varepsilon}_{(i+1)\delta}-X^{\varepsilon}_{i\delta}\right)^2}{\alpha(H^*(m), \varepsilon)^2 \delta}\alpha(H^*(m), \varepsilon)^2 \delta} = \frac{A(\varepsilon, \delta/2)}{A(\varepsilon,\delta)2}
\end{align}
and we know that \(A(\varepsilon,\delta)\) converges in probability to \(\int_0^1G(Y^{\varepsilon}_s)^2ds\) as \(\delta\rightarrow 0\) under the assumptions in the statement of the proposition. Using the Continuous Mapping Theorem for convergence in 
probability \cite{billingsleyConvergenceProbabilityMeasures1999} the ratio then converges to \(\frac{1}{2}\) and the result follows from applying the logarithm (once again under the umbrella of the Continuous Mapping Theorem for convergence in probability). Note that the assumption \(\mathbb{P}(G(Y_t)=0)=0\) is here necessary to apply the Continuous Mapping Theorem. 
\end{proof}
\section{Asymptotic unbiasedness under subsampling}
\label{sec: asymptotic unbiasedness}
The goal of this section is a first step towards the appropriateness of \(C_{\varepsilon, \delta}\) as an estimator: we prove that \(\mathbb{E}[C_{\varepsilon,\delta}]\rightarrow C^2\) as \(\delta/\varepsilon\rightarrow 0\). This is the least we may expect from an estimator and will be a building block for the consistency and distributional results of §§\ref{sec: consistency}--\ref{sec: asymptotic distribution}. The behaviour splits naturally into the two generic regimes \(H^*(m)<1/2\) and \(H^*(m)>1/2\), which we treat together in Theorem~\ref{thm: asymptotic unbiasedness}, and the borderline regime \(H^*(m)=1/2\), which requires an additional logarithmic condition on the subsampling rate and is collected at the end of the section in Theorem~\ref{thm: borderline}.
\begin{theorem}\label{thm: asymptotic unbiasedness}
Let $X^\varepsilon$ be the solution of \eqref{eq: multiscale system} with \(0<H<1\) and \(G\in L^2(\mathbb{R}, \gamma)\) with Hermite rank \(m\) such that \(H^*(m)\neq 1/2\). Let $T=N \delta$ for any time horizon $T$. If $\delta(\varepsilon)>0$ is such that \(\delta(\varepsilon)\xrightarrow{\varepsilon\rightarrow 0} 0\) and \(\varepsilon/\delta(\varepsilon)\rightarrow 0\) then
\[\mathbb{E}[C_{\varepsilon, \delta}]\rightarrow C^2,\]
where
\[C^2 = \mathbb{E}[X_1^2] = \begin{cases}
\displaystyle\sum_{q=m}^{\infty}c_q^2 q!\int_{\mathbb{R}}\rho(s)^q\,ds, & H^*(m)<1/2,\\[1em]
\displaystyle\frac{c_m^2 m!\,(\sigma^2H(2H-1))^m}{H^*(m)(2H^*(m)-1)}, & H^*(m)>1/2.
\end{cases}\]
Quantitatively,
\begin{equation}
    \label{eq: quantitative estimate of expectation}
\mathbb{E}[C_{\varepsilon,\delta}] = C^2 +
\begin{cases}
K\,(\varepsilon/\delta)^{1-(2H^*(m)\vee 0)}, & H^*(m)<1/2,\\[0.6em]
K_1(\varepsilon/\delta)^{2H^*(m)-1}+K_2(\varepsilon/\delta)^{2H^*(m)}, & H^*(m)>1/2,
\end{cases}
\end{equation}
for finite constants \(K, K_1, K_2\) depending only on \(H, m, G\).
\end{theorem}
\begin{proof}
The two cases share the common starting point
\[\mathbb{E}[(X^\varepsilon_\delta)^2] = \alpha(H^*(m),\varepsilon)^2\int_{-\delta/\varepsilon}^{\delta/\varepsilon}\left(\delta/\varepsilon-|s|\right)R(|s|)\,ds,\]
where \(R(s)=\mathbb{E}[G(Y_0)G(Y_s)]=\sum_{q\geq m}c_q^2 q!\rho(s)^q\); the regimes diverge in how the resulting tail integrals behave.

\medskip\noindent\textbf{Case \(H^*(m)<1/2\).}
We first prove the following estimate on the second moment of \(\frac{1}{\sqrt{\varepsilon}}\int_0^{\delta}G(Y^{\varepsilon}_t)dt\).
Note that the estimate is very similar to (4.12) in \cite{gehringerFunctionalLimitTheorems2022a} as it is essentially finding the gap in the inequality for the case \(p=2\). 

Let
\[R(s) = \mathbb{E}[G(Y_0)G(Y_s)] = \sum_{q=m}^{\infty}c_q^2q!\rho(s)^q\]
and note that on the one hand \[|R(s)|\leq R(0) = K\sum_{q=m}^{\infty}c_q^2q! = K ||G||_{L^2(\gamma)}^2<\infty.\] On the other hand, since \(\rho(s)^q\sim \kappa_H^q|s|^{(2H-2)q}\) when \(|s|\) is large, we also have that \(R(s)\sim \kappa_H^m|s|^{m(2H-2)}\) for large \(|s|\). Loosely speaking, \(R(s)\) behaves like a constant close to zero and decays at algebraic rate \(m(2H-2)\) away from zero. We now claim that 
\begin{equation}
\label{eq: order of second moment}
\mathbb{E}[(X^{\varepsilon}_{\delta}-X^{\varepsilon}_0)^2] = C^2\delta + R(\varepsilon,\delta)
\end{equation}
where the remainder is of order \(|R(\varepsilon,\delta)|\sim \varepsilon(\delta/\varepsilon)^{2H^*(m)\vee 0}\).
We note that the remainder is in fact negative (as it may seem as if \eqref{eq: order of second moment} contradicts (4.12) in \cite{gehringerFunctionalLimitTheorems2022a} otherwise) 
but we are just interested in the order. To prove that this is the case we first write it as follows 
\begin{align}
\mathbb{E}[(X^{\varepsilon}_{\delta}-X^{\varepsilon}_0)^2] & = \frac{1}{\varepsilon}\int_0^\delta\int_0^\delta\mathbb{E}[G(Y^{\varepsilon}_t)G(Y^{\varepsilon}_s)]dsdt \\
&= \varepsilon\int_0^{\delta/\varepsilon}\int_0^{\delta/\varepsilon}\mathbb{E}[G(Y_t)G(Y_s)]dsdt \\
&=\varepsilon\int_0^{\delta/\varepsilon}\int_0^{\delta/\varepsilon} R(|s-t|)dsdt \\
&=\varepsilon\int_{-\delta/\varepsilon}^{\delta/\varepsilon} (\delta/\varepsilon - |s|)R(|s|)ds\\
&= \delta\int_{-\delta/\varepsilon}^{\delta/\varepsilon} R(|s|)ds - \varepsilon\int_{-\delta/\varepsilon}^{\delta/\varepsilon} |s|R(|s|)ds
\end{align}
where the fourth equality follows from a change of variables taking advantage of the symmetry of the covariance function. This allows to split the second moment in three terms as follows 
\begin{equation}
\mathbb{E}[(X^{\varepsilon}_{\delta}-X^{\varepsilon}_0)^2] = \delta \int_{\mathbb{R}}R(s)ds - \varepsilon\int_{-\delta/\varepsilon}^{\delta/\varepsilon}|s|R(|s|)ds - \delta\int_{\mathbb{R}\setminus(-\frac{\delta}{\varepsilon},\frac{\delta}{\varepsilon})}R(s)
\end{equation}
We now prove that the second and third term vanish under appropriate subsampling and determine its order. Using the symmetry of the integrand the second term may be written as 
\begin{equation}
\varepsilon\int_{|y|<\delta/\varepsilon}|s|\sum_{q=m}^{\infty}c_q^2q!\rho(|s|)^q ds = 2\varepsilon\int_0^{\delta/\varepsilon}s\sum_{q=m}^{\infty}c_q^2q!\rho(s)^q ds 
\end{equation}
recall that \(\rho(s)\leq 1\wedge|s|^{2H-2}\) hence, assuming \(\delta/\varepsilon>1\) we can split the integral as 
\[
2\varepsilon + 2\varepsilon\int_1^{\delta/\varepsilon}s^{(2H-2)m+1}ds = 2\varepsilon + 2\varepsilon\int_1^{\delta/\varepsilon}s^{2H^*(m)-1}ds
\]
now note that when \(H^*(m)<0\), \(s^{2H^*(m)-1}\) is integrable on the half axis hence the integral in the second term becomes constant and the order is just \(\varepsilon\). When this is not the case we get 
\[
\int_1^{\delta/\varepsilon}s^{2H^*(m)-1}ds = \frac{1}{2H^*(m)}\left((\delta/\varepsilon)^{2H^*(m)} -1\right)
\]
and therefore the order of the second term is \(\varepsilon(\delta/\varepsilon)^{2H^*(m)\vee 0}\)
For the third term we simply use that the first term in the sum dominates and hence determines the order and 
\[
2\delta \int_{\delta/\varepsilon}^{\infty}s^{(2H-2)m} = 2\delta \int_{\delta/\varepsilon}^{\infty}s^{2H^*(m)-2} = \frac{2\delta}{2H^*(m)-1}(\delta/\varepsilon)^{2H^*(m)-1} \propto \varepsilon(\delta/\varepsilon)^{2H^*(m)}
\]
Finally, taking expectation and accounting for the extra term \(\delta^{-1}\) after taking sums we get 
\[\mathbb{E}[C_{\varepsilon,\delta}] = C^2 + (\varepsilon/\delta)^{1-2H^*(m)\vee 0}\]
and hence 
\[
\mathbb{E}[C_{\varepsilon,\delta}]\rightarrow C^2
\]
whenever \(\delta/\varepsilon\rightarrow 0\).

\medskip\noindent\textbf{Case \(H^*(m)>1/2\).}
In this case \(X^{\varepsilon}_t = \varepsilon^{H^*(m)-1}\int_0^tG(Y^{\varepsilon}_s)ds = \varepsilon^{H^*(m)}\int_0^{t/\varepsilon}G(Y_s)ds\).
Let again \(R(s) = \mathbb{E}[G(Y_s)G(Y_0)]\) we have that 
\begin{align}
\label{eq: split H>1/2}
\mathbb{E}[(X^{\varepsilon}_\delta)^2]&=\varepsilon^{2H^*(m)}\int_0^{\delta/\varepsilon}\int_0^{\delta/\varepsilon}R(|s-t|)dsdt \nonumber\\
&= \varepsilon^{2H^*(m)}\int_{-\delta/\varepsilon}^{\delta/\varepsilon}(\delta/\varepsilon - |s|)R(|s|)ds \nonumber\\
&= 2\varepsilon^{2H^*(m)-1}\delta\int_0^{\delta/\varepsilon}R(s)ds - 2\varepsilon^{2H^*(m)}\int_0^{\delta/\varepsilon}sR(s)
\end{align}
Recall that \(R(s) = \sum_{q=m}^{\infty}(c_m)^2m!\rho(s)^{q}\) where \(\rho(s)\sim\sigma^2H(2H-1)|s|^{2H-2}+o(|s|^{2H-4})\) for large \(|s|\) and \(\rho(|s|)\leq 1\). Therefore \(R(s)\sim (c_m)^2m!(\sigma^2H(2H-1)|s|^{2H-2})^m\) and
\begin{align*}
    \int_0^{\delta/\varepsilon}R(s)ds &= \int_0^1R(s)ds + \int_1^{\delta/\varepsilon}R(s)ds \\
    &\sim \frac{c_m^2m!(\sigma^2H(2H-1))^m}{2H^*(m)-1}(\delta/\varepsilon)^{2H^*(m)-1} + K
\end{align*}
where \(K\) is a finite constant. Going back to the first integral in \eqref{eq: split H>1/2} we have that 
\begin{equation*}
2\varepsilon^{2H^*(m)-1}\delta\int_0^{\delta/\varepsilon}R(s)ds \sim \delta^{2H^*(m)}\frac{2c_m^2m!(\sigma^2H(2H-1))^m}{2H^*(m)-1} + \varepsilon^{2H^*(m)-1}\delta K
\end{equation*}
Going now to the second term in \eqref{eq: split H>1/2} following a similar strategy we get 
\begin{align*}
2\varepsilon^{2H^*(m)}\int_0^{\delta/\varepsilon}sR(s)ds &\sim 2\int_0^{s_0}sds + 2c_m^2m!(\sigma^2H(2H-1))^m\int_{s_0}^{\delta/\varepsilon}s^{(2H-2)m+1}ds\\
&=\frac{2c_m^2m!(\sigma^2H(2H-1))^m}{2H^*(m)}\delta^{2H^*(m)} + \varepsilon^{2H^*(m)}K
\end{align*}
Finally, we just need to note that 
\begin{align*}
\mathbb{E}[C_{\varepsilon,\delta}] &= \delta^{-2H^*(m)}\mathbb{E}[(X^{\varepsilon}_{\delta})^2]\\
&= 
\begin{multlined}[t]\frac{2c_m^2m!(\sigma^2H(2H-1))^m}{2H^*(m)-1} - \frac{2c_m^2m!(\sigma^2H(2H-1))^m}{2H^*(m)} + (\varepsilon/\delta)^{2H^*(m)-1}K_1 \\
     + (\varepsilon/\delta)^{2H^*(m)}K_2 \end{multlined}\\
&= \frac{c_m^2m!(\sigma^2H(2H-1))^m}{H^*(m)(2H^*(m)-1)} + (\varepsilon/\delta)^{2H^*(m)-1}K_1 + (\varepsilon/\delta)^{2H^*(m)}K_2\\
&= C + (\varepsilon/\delta)^{2H^*(m)-1}K_1 + (\varepsilon/\delta)^{2H^*(m)}K_2
\end{align*}
which concludes that, when \(\varepsilon/\delta\rightarrow 0\), \(\mathbb{E}[C_{\varepsilon,\delta}]\rightarrow C\). 
\end{proof}
The borderline case \(H^*(m)=1/2\) requires a logarithmic restriction on the subsampling rate; we record it here for completeness, although it does not feature in the consistency or distributional analysis that follows.

\begin{theorem}\label{thm: borderline}
Let $X^\varepsilon$ be the solution of \eqref{eq: multiscale system} with \(0<H<1\), \(G\in L^2(\mathbb{R}, \gamma)\) with Hermite rank \(m\) such that \(H^*(m)=1/2\). Let $T=N \delta$ for any time horizon $T$. If $\delta(\varepsilon)>0$ is such that \(\delta(\varepsilon)\xrightarrow{\varepsilon\rightarrow 0} 0\), \(\varepsilon/\delta(\varepsilon)\rightarrow 0\) and
\(\log(\delta)/\log(\varepsilon)\rightarrow 0\) then \[\mathbb{E}[C_{\varepsilon,\delta}]\rightarrow C\] where \[C^2 = \mathbb{E}[X_1^2] = \frac{c_m^2m!(\sigma^2H(2H-1))^m}{H^*(m)(2H^*(m)-1)}.\] 
In particular, 
\begin{equation}
    \label{eq: quantitative estimate of expectation H=1/2}
\mathbb{E}[C_{\varepsilon,\delta}] = C^2\left|\frac{\log(\delta/\varepsilon)}{\log(\varepsilon)}\right| + K_1\frac{1}{|\log(\varepsilon)|} + K_2 \frac{\varepsilon/\delta}{|\log(\varepsilon)|}
\end{equation}
for some finite constants \(K_1, K_2\).
\end{theorem}
\begin{proof}
In this case \[X^{\varepsilon}_t = \frac{1}{\sqrt{\varepsilon|\log(\varepsilon)|}}\int_0^tG(Y^{\varepsilon}_s)ds = \sqrt{\frac{\varepsilon}{|\log(\varepsilon)|}}\int_0^{t/\varepsilon}G(Y_s)ds.\]
Let again \(R(s) = \mathbb{E}[G(Y_s)G(Y_0)]\) and through the same calculation we get that 
\begin{align}
\label{eq: split H=1/2}
\mathbb{E}[(X^{\varepsilon}_\delta)^2]= 2\frac{1}{|\log(\varepsilon)|}\delta\int_0^{\delta/\varepsilon}R(s)ds - 2\frac{\varepsilon}{|\log(\varepsilon)|}\int_0^{\delta/\varepsilon}sR(s)
\end{align}
The particularity of this case is that \[R(s)\sim (c_m)^2m!(\sigma^2H(2H-1)|s|^{2H-2})^m = (c_m)^2m!(\sigma^2H(2H-1))^m|s|^{-1}\] hence the order of the integrals changes
\begin{align*}
    \int_0^{\delta/\varepsilon}R(s)ds &= \int_0^1R(s)ds + \int_1^{\delta/\varepsilon}R(s)ds \\
    &\sim c_m^2m!(\sigma^2H(2H-1))^m |\log(\delta/\varepsilon)| + K
\end{align*}
where \(K\) is a finite constant. Going back to the first integral in \eqref{eq: split H=1/2} we have that 
\begin{equation*}
2\frac{1}{|\log(\varepsilon)|}\delta\int_0^{\delta/\varepsilon}R(s)ds \sim 2c_m^2m!(\sigma^2H(2H-1))^m \frac{|\log(\delta/\varepsilon)|}{|\log(\varepsilon)|}\delta + \frac{1}{|\log(\varepsilon)|}\delta K
\end{equation*}
Going now to the second term in \eqref{eq: split H=1/2} following a similar strategy we get 
\begin{align*}
2\frac{\varepsilon}{|\log(\varepsilon)|}\int_0^{\delta/\varepsilon}sR(s)ds &\sim 2\int_0^{s_0}sds + 2c_m^2m!(\sigma^2H(2H-1))^m\int_{s_0}^{\delta/\varepsilon}1ds\\
&=2c^2_m m! (\sigma^2H(2H-1))^m \frac{\delta}{\log(\varepsilon)} + \frac{\varepsilon}{|\log(\varepsilon)|}K
\end{align*}
Finally, we just need to note that 
\begin{align*}
\mathbb{E}[C_{\varepsilon,\delta}] &= \delta^{-1}\mathbb{E}[(X^{\varepsilon}_{\delta})^2]\\
&\begin{multlined}
\sim 2c_m^2m!(\sigma^2H(2H-1))^m \frac{|\log(\delta/\varepsilon)|}{|\log(\varepsilon)|} + \frac{1}{|\log(\varepsilon)|} K \\ - 2c^2_m m! (\sigma^2H(2H-1))^m \frac{1}{|\log(\varepsilon)|} - \frac{\varepsilon/\delta}{|\log(\varepsilon)|}K
\end{multlined}\\
&= 2c_m^2m!(\sigma^2H(2H-1))^m \frac{|\log(\delta/\varepsilon)|}{|\log(\varepsilon)|} + K_1\frac{1}{|\log(\varepsilon)|} + K_2\frac{\varepsilon/\delta}{|\log(\varepsilon)|}
\end{align*}
which concludes that, when \(\varepsilon/\delta\rightarrow 0\), \(\log(\delta)/\log(\varepsilon)\rightarrow 0\), then \(\mathbb{E}[C_{\varepsilon,\delta}]\rightarrow C\). 
\end{proof}

\begin{remark}
Note that the standard choice of subsampling rate is \(\delta = \varepsilon^{\alpha}\) for some \(\alpha\in(0,1)\). This choice is not valid for this case and would in fact introduce bias by a factor of \(1-\alpha\). The prototypical choice 
of a subsampling rate in this case would be \(\delta = \exp(-|\log(\varepsilon)|^{\alpha})\) for some \(\alpha\in(0,1)\). 
\end{remark}
This case is not particularly interesting from the inferential point of view. On the one hand, it is very hard (if not impossible) to diagnose whether or not we are in the borderline case purely from a discrete set of observations. This is because 
the parameter \(H^*(m)\) is only "observable" when taking values in \((1/2,1)\). For this reason we exclude this case from the rest of the paper. 
\section{Consistency on a finite Wiener Chaos}
\label{sec: consistency}
\label{sec: consistency}
Now that we have found an observational regime for which the estimator is asymptotically unbiased we want to strengthen this convergence to some mode of consistency. This is much more desirable for our purposes since the setting we want to work in 
is that of only one string of data being observable at a time. 

For this, we will take great advantage of the \(L^2\) structure the estimator has. In particular, in the simple case that \(G(x)=c_mH_m(x)\) the
estimator \(C_{\varepsilon, \delta}\) admits an explicit decomposition into orthogonal Wiener chaoses which allows to determine the asymptotic behaviour of its variance and hence to conclude consistency in \(L^2\).

\subsection*{Chaos decomposition of the estimator}

\begin{proposition}\label{prop: chaos decomposition}
The estimator \(C_{\varepsilon, \delta}\) admits the following decomposition into orthogonal Wiener chaoses
\begin{equation}
    \label{eq: chaos decomposition finite}
C_{\varepsilon, \delta} = c_m^2T_0 + \sum_{r=0}^{m-1} \left(\frac{c_m m!}{(m-r)!\sqrt{r!}}\right)^2T_{2m-2r}
\end{equation}
where the index \(2m-2r\) records the order of the Wiener chaos to which each term \(T_{2m-2r}\) belongs, and
\begin{multline}
T_{2m-2r} = \\  N^{2H'-1}\alpha(\varepsilon,H^*(m))^2I_{2m-2r}\left(\sum_{k=0}^{N-1}\int_{k\delta}^{(k+1)\delta}(\tau_s K^{\varepsilon})^{\otimes m}ds\otimes_r \int_{k\delta}^{(k+1)\delta}(\tau_s K^{\varepsilon})^{\otimes m}ds\right),
\end{multline}
and
\begin{equation}
c_m^2T_0 = \mathbb{E}[C_{\varepsilon, \delta}].
\end{equation}
The combinatorial weights \(\bigl(\tfrac{c_m m!}{(m-r)!\sqrt{r!}}\bigr)^2\) are the normalisation constants from the product formula for multiple Wiener--Itô integrals.
\end{proposition}
\begin{proof}
The decomposition relies on the product formula for multiple Wiener--Itô integrals \eqref{eq: product structure} together with a stochastic Fubini theorem (see \cite[Theorem 2.1]{pipirasRegularizationIntegralRepresentations2010}) which allows us to interchange the deterministic time integral with the multiple Wiener--Itô integral. We start by rewriting \(C_{\varepsilon, \delta}\) as
\begin{align}
C_{\varepsilon, \delta} &= N^{2H'-1}\alpha(\varepsilon,H^*(m))^2\sum_{k=0}^{N-1}\left(\int_{k\delta}^{(k+1)\delta}c_mH_m(Y^{\varepsilon}_s)ds\right)^2\\
&= N^{2H'-1}\alpha(\varepsilon,H^*(m))^2\sum_{k=0}^{N-1}\left(\int_{k\delta}^{(k+1)\delta}c_mI_m((\tau_s K^{\varepsilon})^{\otimes m})ds\right)^2\\
&= N^{2H'-1}\alpha(\varepsilon,H^*(m))^2c_m^2\sum_{k=0}^{N-1}I_m\left(\int_{k\delta}^{(k+1)\delta}(\tau_s K^{\varepsilon})^{\otimes m}ds\right)^2\\
&= 
\begin{multlined}[t]
N^{2H'-1}\alpha(\varepsilon,H^*(m))^2c_m^2\\
\sum_{r=0}^{m}r!\binom{m}{r}^2I_{2m-2r}\left(\sum_{k=0}^{N-1}\int_{k\delta}^{(k+1)\delta}(\tau_s K^{\varepsilon})^{\otimes m}ds\otimes_r \int_{k\delta}^{(k+1)\delta}(\tau_s K^{\varepsilon})^{\otimes m}ds\right)
\end{multlined}
\label{eq: chaos decomposition before change}
\end{align}
\end{proof}
We will denote the kernel of the multiple Wiener-Itô integral above as 
\begin{equation}
g_{k}^{m,r,\varepsilon} = \int_{k\delta}^{(k+1)\delta}(\tau_s K^{\varepsilon})^{\otimes m}ds\otimes_r \int_{k\delta}^{(k+1)\delta}(\tau_s K^{\varepsilon})^{\otimes m}ds
\end{equation}
\begin{corollary}[Variance decomposition]\label{cor: variance decomposition}
The variance of the estimator \(C_{\varepsilon, \delta}\) may be expressed as
\begin{equation}
\mathbb{E}[(C_{\varepsilon, \delta}-\mathbb{E}[C_{\varepsilon, \delta}])^2] = \sum_{r=0}^{m-1} \left(\frac{c_m m!}{(m-r)!\sqrt{r!}}\right)^4\mathbb{E}[T_{2m-2r}^2]
\end{equation}
and the second moment of the chaos terms may be bounded as
\begin{equation}
    \label{eq: second moment of chaos term}
\mathbb{E}[T_{2m-2r}^2] \leq (2m-2r)!N^{2(2H'-1)}\alpha(\varepsilon,H^*(m))^4(E_{2m-2r}^{\text{Diag}} + E_{2m-2r}^{\text{Off}})
\end{equation} 
with the inequality being an equality if \(r=m-1\) due to the kernels being symmetric. The terms \(E_{2m-2r}^{\text{Diag}}\) and \(E_{2m-2r}^{\text{Off}}\)
may be expressed explicitly as
\begin{align}
E_{2m-2r}^{\text{Diag}} &= \sum_{i=0}^{N-1}\left\langle\ g_{i}^{m,r,\varepsilon},g_{i}^{m,r,\varepsilon}\right\rangle_{L^2(\mathbb{R}^{2m-2r})}\\
&=\sum_{i=0}^{N-1}\begin{multlined}[t]
\delta^4\int_{[0,1]^4}\rho((u-v)\delta/\varepsilon)^r\rho((l-s)\delta/\varepsilon)^r\rho((u-s)\delta/\varepsilon+|i-j|\delta/\varepsilon)^{m-r} \\
 \rho((v-l)\delta/\varepsilon+|i-j|\delta/\varepsilon)^{m-r}dudvdlds
\end{multlined}
\end{align}
\begin{align}
E_{2m-2r}^{\text{Off}} &= \sum_{i,j=0, i\neq j}^{N-1}\left\langle\ g_{i}^{m,r,\varepsilon},g_{j}^{m,r,\varepsilon}\right\rangle_{L^2(\mathbb{R}^{2m-2r})} \\  
&=\sum_{i,j=0, i\neq j}^{N-1}\begin{multlined}[t]
\delta^4\int_{[0,1]^4}\rho((u-v)\delta/\varepsilon)^r\rho((l-s)\delta/\varepsilon)^r\rho((u-s)\delta/\varepsilon+|i-j|\delta/\varepsilon)^{m-r} \\ 
\rho((v-l)\delta/\varepsilon+|i-j|\delta/\varepsilon)^{m-r}dudvdlds
\end{multlined}
\end{align}
\end{corollary}
\begin{proof}
Using the isometry between Wiener chaoses and \(L^2\) spaces and the linearity of the symmetrization operator we get that
\begin{align}
\mathbb{E}[I_{2m-2r}^2] &= (2m-2r)!\left|\left|\sum_{k=0}^{N-1}\int_{k\delta}^{(k+1)\delta}(\tau_s K^{\varepsilon})^{\otimes m}ds\tilde{\otimes}_r \int_{k\delta}^{(k+1)\delta}(\tau_s K^{\varepsilon})^{\otimes m}ds\right|\right|_{L^2(\mathbb{R}^{2m-2r})}^2\\
&\leq (2m-2r)!\left|\left|\sum_{k=0}^{N-1}\int_{k\delta}^{(k+1)\delta}(\tau_s K^{\varepsilon})^{\otimes m}ds \otimes_r \int_{k\delta}^{(k+1)\delta}(\tau_s K^{\varepsilon})^{\otimes m}ds\right|\right|_{L^2(\mathbb{R}^{2m-2r})}^2\\
&= (2m-2r)!\sum_{i,j=0}^{N-1}\left\langle\ g_{i}^{m,r,\varepsilon},g_{j}^{m,r,\varepsilon}\right\rangle_{L^2(\mathbb{R}^{2m-2r})}
\label{eq: inequality symmetrization}
\end{align}
We now compute the \(r\)-contractions explicitly for \(r\in\{0,...,m\}\) as: 
\begin{align}
&g_{k}^{m,r,\varepsilon}(x_1,...,x_{m-r},y_1,...,y_{m-r})\\
&=
\begin{multlined}[t]
\int_{\mathbb{R}^p}\left(\int_{k\delta}^{(k+1)\delta}\prod_{i=1}^{m-r}\tau_u K^{\varepsilon}(x_i)\prod_{i=1}^{r}\tau_u K^{\varepsilon}(s_i)du \cdot\right. \\
\left. \int_{k\delta}^{(k+1)\delta}\prod_{i=1}^{m-r}\tau_v K^{\varepsilon}(y_i)\prod_{i=1}^{r}\tau_v K^{\varepsilon}(s_i)dv\right)ds_1...ds_r
\end{multlined}\\
&= \int_{[k\delta,(k+1)\delta]^2}\prod_{i=1}^{m-r}\tau_u K^{\varepsilon}(x_i)\tau_v K^{\varepsilon}(y_i)\left(\int_{\mathbb{R}^p}\prod_{i=1}^{r}\tau_u K^{\varepsilon}(s_i)\tau_v K^{\varepsilon}(s_i)ds_1...ds_r\right)dudv\\
&=\int_{[k\delta,(k+1)\delta]^2}\prod_{i=1}^{m-r}\tau_u K^{\varepsilon}(x_i)\tau_v K^{\varepsilon}(y_i)\left(\int_{\mathbb{R}}\tau_u K^{\varepsilon}(s)\tau_v K^{\varepsilon}(s)ds\right)^r dudv\\
&=\int_{[k\delta,(k+1)\delta]^2}\prod_{i=1}^{m-r}\tau_u K^{\varepsilon}(x_i)\tau_v K^{\varepsilon}(y_i)\rho^{\varepsilon}(u-v)^rdudv\\
\end{align}
In particular, for \(r=m\) 
\begin{align}
g_{k}^{m,m,\varepsilon} &= \int_{[0,\delta]^2}\rho^{\varepsilon}(u-v)^m dudv. 
\end{align}
In particular, the term \(T_0\) satisfies 
\begin{align}
T_0 = N^{2H'}\alpha(\varepsilon,H^*(m))^2\left(\f{c_m}{m!}\right)^2\int_{[0,\delta]^2}\rho^{\varepsilon}(u-v)^m dudv = \mathbb{E}[C_{\varepsilon,\delta}]\\ 
\end{align}
For \(0\leq r\leq m-1\) we have that the inner product may be computed as 
\begin{align}
& \left\langle\ g_{i}^{m,r,\varepsilon},g_{j}^{m,r,\varepsilon}\right\rangle_{L^2(\mathbb{R}^{2m-2r})}\\
&= \begin{multlined}[t]
\int_{\mathbb{R}^{2m-2r}}\left(\int_{[k\delta,(k+1)\delta]^2}\prod_{i=1}^{m-r}\tau_u K^{\varepsilon}(x_i)\tau_v K^{\varepsilon}(y_i)\rho^{\varepsilon}(u-v)^rdudv\right)\cdot \\ 
\left(\int_{[k\delta,(k+1)\delta]^2}\prod_{i=1}^{m-r}\tau_u K^{\varepsilon}(x_i)\tau_v K^{\varepsilon}(y_i)\rho^{\varepsilon}(u-v)^rdudv\right)dxdy
\end{multlined}\\
&=
\begin{multlined}[t]
\int_{[i\delta,(i+1)\delta]^2}dudv\int_{[j\delta,(j+1)\delta]^2}dlds\rho^{\varepsilon}(u-v)^r\rho^{\varepsilon}(l-s)^r\\
\int_{\mathbb{R}^{2m-2r}}\prod_{i=1}^{m-r}\tau_u K^{\varepsilon}(x_i)\tau_v K^{\varepsilon}(y_i)\prod_{i=1}^{m-r}\tau_s K^{\varepsilon}(x_i)\tau_l K^{\varepsilon}(y_i)dxdy
\end{multlined}\\
&=\int_{[i\delta,(i+1)\delta]^2}dudv\int_{[j\delta,(j+1)\delta]^2}dlds\quad \rho^{\varepsilon}(u-v)^r\rho^{\varepsilon}(l-s)^r\rho^{\varepsilon}(u-s)^{m-r}\rho^{\varepsilon}(v-l)^{m-r}\\
&=
\begin{multlined}[t]
\delta^4\int_{[0,1]^4}\rho((u-v)\delta/\varepsilon)^r\rho((l-s)\delta/\varepsilon)^r \\ 
\rho((u-s)\delta/\varepsilon+|i-j|\delta/\varepsilon)^{m-r}\rho((v-l)\delta/\varepsilon+|i-j|\delta/\varepsilon)^{m-r}dudvdlds
\end{multlined}
\end{align}
\end{proof}

We devote the rest of the section to determining the asymptotic orders of \(E_{2m-2r}^{\text{Diag}}\) and \(E_{2m-2r}^{\text{Off}}\) in the observational regime in which the estimator is asymptotically unbiased. This is a building block towards both extending the consistency result to more general functions \(G\) and to establishing the fluctuating behaviour of the estimator for uncertainty quantification. The two regimes \(H^*(m)>1/2\) and \(H^*(m)<1/2\) are qualitatively different: in the former, only the second chaos \(T_2\) contributes at leading order; in the latter, the chaos terms split into qualitatively different regimes governed by two critical exponents \(r_I, r_S\) introduced in \eqref{eq: rI rS def} below.

\subsection*{Variance asymptotics}

\begin{theorem}[Variance bound on a finite chaos]\label{thm: variance bound}
Let \(\delta(\varepsilon)>0\) be such that \(\delta(\varepsilon)\xrightarrow{\varepsilon\to 0} 0\) and \(\varepsilon/\delta\to 0\) as \(\varepsilon\to 0\). For \(\varepsilon\) small enough,
\[
\mathrm{Var}(C_{\varepsilon, \delta}) \;\lesssim\;
\begin{cases}
\delta^{2(2-2H)\wedge 1}, & H^*(m)>1/2,\\[0.6em]
\delta + \varepsilon^{2(2-2H)}, & H^*(m)<1/2,
\end{cases}
\]
with implicit constants depending only on \(H, m, G\). In both regimes \(\mathrm{Var}(C_{\varepsilon,\delta})\to 0\) as \(\varepsilon\to 0\), so that \(C_{\varepsilon,\delta}\to \mathbb{E}[C_{\varepsilon,\delta}]\) in \(L^2\).
\end{theorem}

The chaos-by-chaos asymptotics behind Theorem~\ref{thm: variance bound} are recorded in Propositions~\ref{prop: chaos asymptotics H>1/2} and~\ref{prop: chaos asymptotics H<1/2} below. Their proofs reduce, via Corollary~\ref{cor: variance decomposition}, to estimating \(E^{\mathrm{Diag}}_{2m-2r}\) and \(E^{\mathrm{Off}}_{2m-2r}\) on the relevant range of \(r\); the technical asymptotic-bound lemmas underlying these estimates are collected in Appendix~\ref{sec: proofs of consistency}.

\begin{proposition}[Chaos asymptotics, \(H^*(m)>1/2\)]\label{prop: chaos asymptotics H>1/2}
Let \(\delta(\varepsilon)>0\) be such that \(\delta(\varepsilon)\xrightarrow{\varepsilon\rightarrow 0} 0\) and \(\varepsilon/\delta\rightarrow 0\) as \(\varepsilon\rightarrow 0\). For \(\varepsilon\) small enough and \(H\neq 3/4\) we have the asymptotic behaviour
\begin{align}
    \label{eq: asymptotic finite chaoses H>1/2}
\mathbb{E}[T_{2}^2]&\sim D_{\star}(H,m)\,\delta^{2(2-2H)\wedge 1},\\ 
\mathbb{E}[T_{2m-2r}^2]&\leq \bigl(D^{\mathrm{Diag}}_{H,m,r}+D^{\mathrm{Off}}_{H,m,r}\bigr)\delta^{2(m-r)(2-2H)\wedge 1}\quad 1\leq r< m-1,
\end{align}
where \(D_{\star}(H,m)\) is defined in \eqref{eq: D star} and the prefactors \(D^{\mathrm{Diag}}_{H,m,r}, D^{\mathrm{Off}}_{H,m,r}>0\), depending only on \(H,m\) and \(r\), are read off Lemmas~\ref{lem: asymptotic variance m=1 H>1/2}--\ref{lem: asymptotic variance m>1 H>1/2}; only \(D_{\star}(H,m)\) is used by name in the rest of the paper. The borderline value \(H=3/4\) carries an additional \(|\log\delta|\) factor in the second-chaos asymptotic. In particular,
\begin{equation}
\label{eq: variance bound finite chaos H>1/2}
\mathbb{E}[(C_{\varepsilon, \delta}-\mathbb{E}[C_{\varepsilon, \delta}])^2]\sim D_{\star}(H,m)\,\delta^{2(2-2H)\wedge 1}.
\end{equation}
\end{proposition}

When \(H^*(m)<1/2\) the chaos terms are governed by two critical values of \(r\) that determine which of the integrals and Riemann sums underlying the asymptotics are convergent and which diverge. We define
\begin{align}
    \label{eq: rI rS def}
    r_I &= \min\{0\leq r \leq m: r(2H-2)<-1\}\\
    r_S &= \min\{\lfloor m/2\rfloor\leq r \leq m: 2(m-r)(2H-2)>-1\}\nonumber
\end{align}
For the integrability note that if \(r>r_I\) then \(r(2H-2)<-1\) and the integral \(\int_0^{\infty}\rho(s)^r ds<\infty\), whereas if \(r\leq r_I\) the integral diverges and we need to account for its order. For the Riemann sum, if \(r>r_S\) then \(2(m-r)(2H-2)>-1\) and \(\int_0^1(1-x)x^{2(m-r)(2H-2)}dx<\infty\), whereas if \(r\leq r_S\) the integral diverges and we again need its order.

These two values do not always appear: \(r_S=m\) unless \(m>2\) and \(H\in(3/4,1-1/(2m))\), in which case \(2(2H-2)>-1\) and we have at least \(m-1\leq r_S\); for \(r_I\) we have that \(r_I=0\) if \(H<1/2\) and \(r_I>0\) if \(H>1/2\) since \(2H-2>-1\) in the latter case. To summarize:
\begin{itemize}
    \item If \(H<1/2\) then \(r_I=0\), \(r_S=m\) and all the terms in the chaos decomposition have the same asymptotic behaviour.
    \item If \(H\in(1/2,3/4)\) then \(r_S = m\) but \(r_I>0\).
    \item If \(H\in(3/4, 1-1/(2m))\) then \(r_S<m\) and \(r_I>0\).
\end{itemize}
Note that if \(H\geq 1 - 1/(2m)\) then \(H^*(m)\geq 1/2\). We now compare the two in the latter case.
\begin{proposition}\label{prop: rI rS comparison}
Suppose that \(m>2\) and \(H\in(3/4,1-1/(2m))\). Then:
\begin{itemize}
    \item If \(H^*(m)>1/4\), then \(r_I\geq r_S\).
    \item If \(H^*(m)<1/4\), then \(r_I\leq r_S\).
    \item If \(H^*(m)=1/4\) then \(r_I=r_S\).
\end{itemize}
\end{proposition}
\begin{proof}
\textit{Case 1:} assume that \(H^*(m)>1/4\) and let \(r>r_I\) then \(-r(2H-2)>1\) and
\begin{align}
    2(m-r)(2H-2) &= 4H^*(m) - 4 - 2r(2H-2)\\
&> 4H^*(m) - 4 + 2 = 4H^*(m) -2 >-1\\
\end{align}
which implies that \(r>r_S\) and hence \(r_I\geq r_S\).

\textit{Case 2:} assume that \(H^*(m)<1/4\) and let \(r>r_S\) then \(-(m-r)(2H-2)<1/2\) and
\begin{align}
    r(2H-2) &= 2H^*(m) -2 - (m-r)(2H-2)\\
&< 2H^*(m) -2 + 1/2 = 2H^*(m) - 3/2 < -1
\end{align}
which implies that \(r>r_I\) and hence \(r_I\leq r_S\).

\textit{Case 3:} if \(H^*(m) = 1/4\) then using the fact that
\begin{equation}
2(m-r)(2H-2) + 2r(2H-2) = 4H^*(m) -4 = -3
\end{equation}
we have that \(2(m-r)(2H-2)>-1\) if and only if \(r(2H-2)<-1\) and hence \(r_I=r_S\).
\end{proof}

\begin{proposition}[Chaos asymptotics, \(H^*(m)<1/2\)]\label{prop: chaos asymptotics H<1/2}
Let \(\delta(\varepsilon)>0\) be such that \(\delta(\varepsilon)\xrightarrow{\varepsilon\rightarrow 0} 0\) and \(\varepsilon/\delta\rightarrow 0\) as \(\varepsilon\rightarrow 0\). For \(\varepsilon\) small enough, for any \(0\leq r\leq m-1\),
\begin{equation}
    \label{eq: asymptotic finite chaoses H<1/2}
\mathbb{E}[T_{2m-2r}^2]\leq D^{\mathrm{Diag}}_{H,m,r}\delta+D^{\mathrm{Off}}_{H,m,r}f(\delta,\varepsilon,m,r),
\end{equation}
with \(D^{\mathrm{Diag}}_{H,m,r}, D^{\mathrm{Off}}_{H,m,r}>0\) read off Lemmas~\ref{lem: second chaos m=1 H<1/2}--\ref{lem: asymptotics m>1 H<1/2}, and
\begin{equation}
\label{eq: f case table H<1/2}
f(\delta, \varepsilon, m,r) =
\begin{cases}
\delta(\varepsilon/\delta)^{2m(2-2H)-2} \quad &r< r_I\wedge r_S\\
\varepsilon^{2(m-r)(2-2H)} \quad &r> r_I\vee r_S\\
\delta^{2(m-r)(2H-2)}(\varepsilon/\delta)^{2m(2-2H)-2} \quad &r_S< r < r_I\\
\delta(\varepsilon/\delta)^{2(m-r)(2-2H)} \quad &r_I< r< r_S
\end{cases}
\end{equation}
The borderline cases \(r(2H-2)=-1\) or \(2(m-r)(2H-2)=-1\) carry the same orders up to logarithmic corrections: when \(r=r_I\) and \(r(2H-2)=-1\),
\[
f(\delta,\varepsilon,m,r) \sim
\begin{cases}
\delta(\varepsilon/\delta)^{2(m-r)(2-2H)}\log(\delta/\varepsilon)^2, & H^*(m)>1/4,\\
\varepsilon^{2(m-r)(2-2H)}\log(\delta/\varepsilon)^2, & H^*(m)<1/4;
\end{cases}
\]
when \(r=r_S\) and \(2(m-r)(2H-2)=-1\),
\[
f(\delta,\varepsilon,m,r) \sim
\begin{cases}
\varepsilon^{2(m-r)(2-2H)}\log(\delta), & H^*(m)>1/4,\\
\log(\delta)(\varepsilon/\delta)^{2m(2-2H)-2}, & H^*(m)<1/4;
\end{cases}
\]
and when \(r=r_I=r_S\) (so that \(H^*(m)=1/4\)), \(f(\delta,\varepsilon,m,r)\sim \varepsilon\log(\delta)\log(\delta/\varepsilon)\). In particular,
\begin{equation}
\label{eq: variance bound finite chaos H<1/2}
\mathbb{E}[(C_{\varepsilon, \delta}-\mathbb{E}[C_{\varepsilon, \delta}])^2]\lesssim \delta+\varepsilon^{2(2-2H)}.
\end{equation}
\end{proposition}
The proof of all the results in this section follows straightforwardly from the asymptotic order of \(E^{\mathrm{Diag}}_{H,m,r}\) and \(E^{\mathrm{Off}}_{H,m,r}\) shown in Appendix~\ref{sec: proofs of consistency} in each of the relevant cases and plugging them into the variance decomposition. \eqref{eq: second moment of chaos term}.
\section{Asymptotic Distribution}
\label{sec: asymptotic distribution}
The fluctuating behaviour of the estimator \(C_{\varepsilon, \delta}\) around the true value has two distinct regimes depending on the amount of correlation present in the underlying model. If it decays fast enough, then we see 
a Central Limit Theorem behaviour and the fluctuations, appropriately renormalized, are Gaussian. If the correlation decays too slowly, then these are not any longer Gaussian but, as in the limiting setting \cite{chronopoulouSelfsimilarityParameterEstimation2011}, 
we observe a Rosenblatt distribution. This resembles very well the classical dichotomy in the Breuer-Major Theorem. Note that these two different regimes can already be deduced from the asymptotic behaviour of the second chaos in the expansion of the estimator. In 
the case where the order is \(\delta\) that means we need to rescale the difference by \(\delta^{-1/2} = \sqrt{N}\) to see non-trivial fluctuations which is a good indicator that we should expect a Central Limit Theorem. The proof of this Central Limit Theorem is current work in progress.
\subsection{Non-Central Limit Case}
\begin{theorem}
    \label{thm: non central limit}
Let \(H>3/4\) and \(H^*(m)>1/2\). If \(C_{\varepsilon,\delta}\) is the estimator defined in \eqref{eq: estimators} then, if \(\delta = \varepsilon^{\alpha}\) for some \(0<\alpha<\frac{2H^*(m)-1}{2H^*(m-1)-1}\) we have that
\begin{equation}
\delta^{2H-2}(C_{\varepsilon,\delta} - C) \xrightarrow[\varepsilon\rightarrow 0]{L^2} c_H Z^{2,2H-1}_1
\end{equation}
where \(Z^{2,2H-1}_1\) is the observed value at time \(1\) of the Rosenblatt process constructed with the same Brownian motion that generates \(y^{\varepsilon}\) and \(c_H\) is an explicit constant. 
\begin{multline}
\label{eq: constant non central limit}
c_H = \left(\frac{c_m m!}{\sqrt{(m-1)!}}\right)^2\frac{(1/\Gamma(2H-1))^m}{((m-1)(2H-2)+2)((m-1)(2H-2)+1)} \\ \left(\frac{8}{(4H-3)(2H-1)}\right)^{1/2}
\end{multline}
\end{theorem}
The bulk of the proof is contained in the following Lemma, whose proof we postpone to an appendix section due to its technical involvement.
\begin{lemma}
\label{lem: convergence of the inner product}
Let \[\tilde{f}_{\varepsilon, \delta}(x,y) = N^{2H^*(m)-1 + (2-2H)}\varepsilon^{2(H^*(m)-1)}\sum_{k=0}^{N-1}g_k^{m,m-1,\varepsilon}(x,y)\] and \[f_t(x,y) =\frac{K(2H-1,2)}{2}\int_0^t (u-x)_{+}^{H-3/2}(u-y)_{+}^{H-3/2}du\] then, if \(\delta = 1/N\) is chosen such that \(\delta = \varepsilon^{\alpha}\) 
for some \(0<\alpha<1\) it holds that 
\begin{multline}
\langle \tilde{f}_{\varepsilon, \delta}, f_1\rangle_{L^2(\mathbb{R}^2)} \xrightarrow[\varepsilon\rightarrow 0]{} \\ \frac{2(1/\Gamma(2H-1))^{m}}{((m-1)(2H-2)+2)((m-1)(2H-2)+1)\left((4H-3)(2H-1)2\right)^{1/2}}
\end{multline}
\end{lemma}
\begin{proof}[of Theorem \ref{thm: non central limit}]
We first decompose the quantity of interest as 
\begin{multline}
\delta^{2H-2}(C_{\varepsilon,\delta} - C) - c_H Z^{2,2H-1}_1 = \\ \delta^{2H-2}(C_{\varepsilon,\delta} - \mathbb{E}[C_{\varepsilon,\delta}]- c_HZ^{2,2H-1}_1) + \delta^{2H-2}(\mathbb{E}[C_{\varepsilon,\delta}] - C)
\end{multline}
We first show that the second term converges to zero using the previous asymptotic unbiasedness result. In particular, from \eqref{eq: quantitative estimate of expectation} we have that
\begin{align}
\delta^{2H-2}|\mathbb{E}[C_{\varepsilon,\delta}] - C| &\sim K \delta^{2H-2}(\varepsilon/\delta)^{2H^*(m)-1}\\
& = K \varepsilon^{2H^*(m)-1}\delta^{2H-2 + m(2-2H)-1}\\
& = K \varepsilon^{2H^*(m)-1}\delta^{1-2H^*(m-1)}
\end{align}
and the right-hand side vanishes by the assumption on \(\alpha\). Recall that \(C_{\varepsilon,\delta} - \mathbb{E}[C_{\varepsilon,\delta}] = \sum_{r=1}^{2m}\left(\frac{c_m m!}{(m-r)!\sqrt{r!}}\right)^2T_{2m-2r}\) where the decomposition is orthogonal in \(L^2\). We further decompose this term as 
\begin{align}
\delta^{2H-2}(C_{\varepsilon,\delta} - \mathbb{E}[C_{\varepsilon,\delta}]- c_HZ^{2,2H-1}_1) & = (\delta^{2H-2}T_{2} - c_HZ^{2,2H-1}_1) + \sum_{r=1}^{m-1}\delta^{2H-2}T_{2m-2r}
\end{align}
to then show that each term converges to zero. For the second term we have from Proposition \ref{prop: chaos asymptotics H>1/2} that
\begin{equation}
\delta^{2(2H-2)}\mathbb{E}[T_{2m-2r}^2] \lesssim
\begin{cases}
\delta^{2(m-r-1)(2-2H)} & r>r_I\\
\delta^{1-2(2-2H)}|\log(\delta)| & r=r_I\\
\delta^{1-2(2-2H)} & r<r_I    
\end{cases}
\end{equation}
and they all vanish as \(H>3/4\) and \(r<m\). The main step is to show that the first term converges to zero. To do this we will take advantage of the isometry property between Wiener chaoses and the space of square-integrable symmetric functions. 
Note that 
\begin{equation}
    Z^{2,2H-1}_t = I_2(f_t(x,y))
\end{equation}
where 
\begin{equation}
    f_t(x,y) = \frac{K(2H-1,2)}{2}\int_0^t (u-x)_{+}^{H-3/2}(u-y)_{+}^{H-3/2}du.
\end{equation}
and 
\begin{equation}
    K(2H-1,2) = \left(\frac{2(2H-1)(4H-3)}{\beta(H-1/2,2-2H)^2}\right)^{1/2}.
\end{equation}
On the other hand, recall that 
\begin{align}
\delta^{2H-2}T_2 &= \left(\frac{c_m m!}{\sqrt{(m-1)!}}\right)^2 N^{(2H^*(m)-1)+2-2H}\varepsilon^{2(H^*(m)-1)}I_{2}\left(\sum_{k=0}^{N-1}g_k^{m,m-1,\varepsilon}(x,y)\right) \\ 
&= I_2(f_{\varepsilon,\delta}(x,y))
\end{align}
where 
\begin{equation}
     f_{\varepsilon,\delta}(x,y) = \left(\frac{c_m m!}{\sqrt{(m-1)!}}\right)^2 N^{2H^*(m)-1}N^{2-2H}\varepsilon^{2(H^*(m)-1)}\sum_{k=0}^{N-1}g_k^{m,m-1,\varepsilon}(x,y)
\end{equation}
Therefore, by the isometry property of Wiener chaoses it is equivalent to show that 
\begin{equation}
    \left|\left|c_H^{-1}f_{\varepsilon,\delta} - f_1(x,y)\right|\right|_{L^2(\mathbb{R}^2)} \xrightarrow[\varepsilon\to 0]{}0
\end{equation}
To prove this we use firstly the polarization identity of the Hilbert space \(L^2(\mathbb{R}^2)\): 
\begin{equation}
\left|\left|c_H^{-1}f_{\varepsilon,\delta} - f_1(x,y)\right|\right|_{L^2(\mathbb{R}^2)}^2 = ||c_H^{-1}f_{\varepsilon,\delta}||_{L^2(\mathbb{R}^2)}^2 + ||f_1(x,y)||_{L^2(\mathbb{R}^2)}^2 - 2\langle c_H^{-1}f_{\varepsilon,\delta}, f_1\rangle_{L^2(\mathbb{R}^2)}.
\end{equation}
We will show that both the norms and the inner product converge to the same constant, which concludes the proof. The named symbols \(D_{\star}(H,m), c_H, K(2H-1,2)\) are those of \eqref{eq: D star}, \eqref{eq: constant non central limit} and \eqref{eq: hermite normalisation} respectively. Let \[\tilde{f}_{\varepsilon,\delta}(x,y)=N^{2H^*(m)-1 + 2-2H}\varepsilon^{2(H^*(m)-1)}\sum_{k=0}^{N-1}g_k^{m,m-1,\varepsilon}(x,y).\] We know from \eqref{eq: asymptotic finite chaoses H>1/2} (case \(r=m-1\)) that
\begin{equation}
\|\tilde{f}_{\varepsilon,\delta}\|_{L^2(\mathbb{R}^2)}^2 \xrightarrow[\varepsilon\to 0]{} D_{\star}(H,m),
\end{equation}
where we recall that
\[D_{\star}(H,m) = \frac{4\,(1/\Gamma(2H-1))^{2m}}{(4H-3)(2H-1)\bigl((m-1)(2H-2)+2\bigr)^{2}\bigl((m-1)(2H-2)+1\bigr)^{2}},\]
and so, dividing by \(c_H^2\) with
\begin{equation}
    \label{eq: cH explicit reminder}
c_H^2 = 2 D_{\star}(H,m)\left(\frac{c_m m!}{\sqrt{(m-1)!}}\right)^4 \quad\text{(by \eqref{eq: constant non central limit}),}
\end{equation}
it follows that
\begin{equation}
||c_H^{-1}f_{\varepsilon,\delta}||_{L^2(\mathbb{R}^2)}^2 \xrightarrow[\varepsilon\to 0]{} 1/2.
\end{equation}
It is easy to see that 
\begin{align*}
&||f_1||_{L^2(\mathbb{R}^2)}^2 = \frac{K(2H-1,2)^2}{4}\int_{\mathbb{R}^2}\left(\int_0^1 (u-x)_{+}^{H-3/2}(u-y)_{+}^{H-3/2}du\right)^2 dxdy \\
&\begin{multlined}
=\frac{K(2H-1,2)^2}{4}\int_0^1\int_0^1\left(\int_{-\infty}^{u\vee v}(u-x)_{+}^{H-3/2}(v-x)_{+}^{H-3/2}dx\right)\\\left(\int_{-\infty}^{u\vee v}(u-y)_{+}^{H-3/2}(v-y)_{+}^{H-3/2}dy\right)dudv 
\end{multlined} \\
&=\frac{K(2H-1,2)^2}{4}\beta(H-1/2,2-2H)^2\int_{[0,1]^2}|u-v|^{4H-4}dudv\\
&= \frac{\beta(H-1/2,2-2H)^2}{(4H-3)(2H-1)}\frac{(4H-3)(2H-1)}{2\beta(H-1/2,2-2H)^2}= 1/2
\end{align*}
which agrees with the variance of the Rosenblatt process at time \(1\) (which is \(1\) after multiplied by the \(2!\) factorial that accounts for symmetries). Now the last step is to show that 
\begin{equation}
\langle c_H^{-1}f_{\varepsilon,\delta}, f_1\rangle_{L^2(\mathbb{R}^2)} \xrightarrow[\varepsilon\to 0]{} 1/2
\end{equation} 
For this we simply note that, by Lemma \ref{lem: convergence of the inner product} we have that 
\begin{equation}
\langle \tilde{f}_{\varepsilon,\delta}, f_1\rangle_{L^2(\mathbb{R}^2)} \xrightarrow[\varepsilon\to 0]{} R
\end{equation} 
where 
\begin{equation}
R = \frac{2(1/\Gamma(2H-1))^{m}}{((m-1)(2H-2)+2)((m-1)(2H-2)+1)\left((4H-3)(2H-1)2\right)^{1/2}},
\end{equation}
therefore
\begin{equation}
\langle f_{\varepsilon,\delta}, f_1\rangle_{L^2(\mathbb{R}^2)} \xrightarrow[\varepsilon\to 0]{} \left(\frac{c_m m!}{\sqrt{(m-1)!}}\right)^2R.
\end{equation}
Combining \eqref{eq: constant non central limit} for \(c_H\) with the explicit form of \(R\) above, the prefactors \((1/\Gamma(2H-1))^m\) and \(((m-1)(2H-2)+2)((m-1)(2H-2)+1)\) cancel and one is left with
\begin{equation}
    \label{eq: rosenblatt cancellation}
\left(\frac{c_m m!}{\sqrt{(m-1)!}}\right)^2 R \;=\; \frac{c_H}{2},\qquad\text{hence}\qquad\langle c_H^{-1}f_{\varepsilon,\delta}, f_1\rangle_{L^2(\mathbb{R}^2)}\xrightarrow[\varepsilon\to 0]{} \tfrac{1}{2},
\end{equation}
matching the limits \(1/2\) for \(\|c_H^{-1}f_{\varepsilon,\delta}\|^2\) and \(\|f_1\|^2\). This is the consistency of the constants which the polarization identity requires.
\end{proof}

\section{Application to 1d fluctuation problems}
An application to have in mind of the previous results \eqref{eq: multiscale system} is to perform statistical inference on the effective diffusivity of random ODEs of the form 
\begin{equation}
\label{eq: random odes}
\dot{x}_t^{\varepsilon} = h(x_t^{\varepsilon})g(y_t^{\varepsilon}) + \alpha(\varepsilon)G(y_t^{\varepsilon}) 
\end{equation}
where \(y_t^{\varepsilon}\) is the fractional Ornstein-Uhlenbeck process defined in \eqref{eq: fOU SDE}, \(G\) is a measurable function that is square-integrable
with respect to the standard Gaussian measure just as in previous sections and \(h\in\mathcal{C}_b^2(\mathbb{R},\mathbb{R})\), \(g\in\mathcal{C}_b(\mathbb{R},\mathbb{R})\), \(H\in(1/3,1)\). 
As shown in \cite{gehringerFunctionalLimitTheorems2022a} these equations have effective dynamics as \(\varepsilon\to 0\) described as solutions to stochastic differential equations 
driven by either Brownian motion or Hermite processes of the form 
\begin{equation}
d\bar{x}_t = \bar{g}h(\bar{x}_t)dt + C\circ \ dW_t
\end{equation}
or 
\begin{equation}
d\bar{x}_t = \bar{g}h(\bar{x}_t)dt + C dZ_t^{H,m}
\end{equation}
where \(\bar{g} = \int_{\mathbb{R}}g(y)d\mu(y)\) and the stochastic integrals are to be understood in the Stratonovich and Young sense respectively (clearly it does not matter in this case of additive noise but the noise could be
multiplicative in which case the interpretation matters, see \cite{gehringerFunctionalLimitTheorems2022a}).

We show that the estimator \(C_{\varepsilon,\delta}\) defined in \eqref{eq: estimators} may be used to consistently estimate the constant \(C\) appearing in the limiting equations above. In other words, that bounded drifts do not alter the \(L^2\)-consistency
of the estimator.

We first derive the estimate 
\begin{proposition}
Let \(x^{\varepsilon}_t\) be the solution to the random ODE \eqref{eq: random odes} and \(y_t^{\varepsilon}\) be the fractional Ornstein-Uhlenbeck process defined in \eqref{eq: fOU SDE}. Then for any \(G\in L^2(\mu)\) we have that 
\begin{equation}
\left\|\alpha(\varepsilon, H^*(m))\int_0^\delta G(Y^{\varepsilon}_s)ds\right\|_{L^2}\sim K\delta^{1/2\vee H^*(m)}
\end{equation}
\end{proposition}
\begin{proof}
Recall that \(\alpha(\varepsilon, H^*(m)) = \varepsilon^{(H^*(m)-1)\vee -1/2}\) and therefore 
\[\alpha(\varepsilon, H^*(m))\int_0^\delta G(Y^{\varepsilon}_s)ds = \varepsilon^{(H^*(m))\vee 1/2}\int_0^{\delta/\varepsilon} G(Y_s)ds\]
where the equality is in law. Using now the same procedure as in the proofs of Section 4 and the estimates from Lemma \ref{lem: integral asymptotics}
we have that 
\begin{align}
\mathbb{E}\left[\left(\int_0^{\delta/\varepsilon} G(Y_s)ds\right)^2\right] &= \int_{[0,\delta/\varepsilon]^2}\mathbb{E}[G(Y_s)G(Y_t)]dsdt\\
&\sim K(\delta/\varepsilon)^{(m(2H-2)+2)\vee 1}\\
&= K (\delta/\varepsilon)^{2H^*(m)\vee 1}
\end{align}
concluding that 
\begin{equation}
\left|\left|\alpha(\varepsilon)\int_0^\delta G(Y^{\varepsilon}_s)ds\right|\right|_{L^2}\sim K\varepsilon^{H^*(m)\vee 1/2}(\delta/\varepsilon)^{H^*(m)\wedge 1} = K\delta^{H^*(m)\vee 1/2}
\end{equation}
\end{proof}
\begin{theorem}
Let \(x^{\varepsilon}_t\) be the solution to the random ODE \eqref{eq: random odes} and \(y_t^{\varepsilon}\) be the fractional Ornstein-Uhlenbeck process defined in \eqref{eq: fOU SDE}. Then the estimator \(C_{\varepsilon, \delta}\) defined in \eqref{eq: estimators} is consistent in \(L^2\) for the constant \(C\)
appearing in the effective dynamics of \(x^{\varepsilon}_t\) as \(\varepsilon\to 0\) if and only if \(\delta(\varepsilon)\) is chosen so that \(\delta\to 0\) and \(\varepsilon/\delta\to 0\).
\end{theorem}
\begin{proof}
Let \(f(\delta) = \delta^{(1-H^*(m))\vee 0}\). The increments may be written as \[x^{\varepsilon}_{u,t} = \int_u^t h(x_s^{\varepsilon})g(y_s^{\varepsilon})ds + \alpha(\varepsilon)\int_u^t G(y_s^{\varepsilon})ds.\]
Note that 
\begin{align}
C_{\varepsilon, \delta} =& f(N)\sum_{k=0}^{N-1}(x^{\varepsilon}_{k\delta, (k+1)\delta})^2 \\
=& f(\delta)\sum_{k=0}^{N-1}\left(\alpha(\varepsilon)\int_{k\delta}^{(k+1)\delta}G(y_s^{\varepsilon})ds\right)^2 + \\
& f(\delta)\sum_{k=0}^{N-1}
\left(\alpha(\varepsilon)\int_{k\delta}^{(k+1)\delta}G(y_s^{\varepsilon})ds\right)\left(\int_{k\delta}^{(k+1)\delta}g(x_s^{\varepsilon})h(y_s^{\varepsilon})ds\right) + \\
& f(\delta)\sum_{k=0}^{N-1}\left(\int_{k\delta}^{(k+1)\delta}g(x_s^{\varepsilon})h(y_s^{\varepsilon})ds\right)^2.
\end{align}
using the boundedness of \(h,g\) we have the deterministic bound 
\[\int_{k\delta}^{(k+1)\delta}g(x_s^{\varepsilon})h(y_s^{\varepsilon})ds \leq ||g||_{\infty}||h||_{\infty}\delta = K\delta\]
hence the last term is controlled by \(Kf(N)\delta=K\delta^{(2-2H^*(m))\vee 1}\). For the cross-term we have the \(L^2\) bound
\begin{align}
&f(N)\sum_{k=0}^{N-1}
\left(\alpha(\varepsilon)\int_{k\delta}^{(k+1)\delta}G(y_s^{\varepsilon})ds\right)\left(\int_{k\delta}^{(k+1)\delta}g(x_s^{\varepsilon})h(y_s^{\varepsilon})ds\right)\\
&\leq 2K\delta \left|\left|f(N)\sum_{k=0}^{N-1}\alpha(\varepsilon)\int_{k\delta}^{(k+1)\delta}G(y_s^{\varepsilon})ds\right|\right|_{L^2} \\ 
&\leq 2K f(N) \left|\left|\alpha(\varepsilon)\int_{k\delta}^{(k+1)\delta}G(y_s^{\varepsilon})ds\right|\right|_{L^2}\\
&\sim 2K\delta^{(1-H^*(m))\vee 1/2} 
\end{align}
which goes to zero as \(\delta\to 0\) for any \(H^*(m)\in(0,1)\). Therefore the consistency results hold true in this setting as well since you can do 
\begin{align}
||C_{\varepsilon, \delta} - C||_{L^2} \leq & \left|\left|f(N)\sum_{k=0}^{N-1}\left(\alpha(\varepsilon)\int_{k\delta}^{(k+1)\delta}G(y_s^{\varepsilon})ds\right)^2  - C\right|\right|_{L^2} +\\ 
& + \left|\left|f(N)\sum_{k=0}^{N-1}
\left(\alpha(\varepsilon)\int_{k\delta}^{(k+1)\delta}G(y_s^{\varepsilon})ds\right)\left(\int_{k\delta}^{(k+1)\delta}g(x_s^{\varepsilon})h(y_s^{\varepsilon})ds\right)\right|\right|_{L^2} \\
& + \left|\left|f(N)\sum_{k=0}^{N-1}\left(\int_{k\delta}^{(k+1)\delta}g(x_s^{\varepsilon})h(y_s^{\varepsilon})ds\right)^2 \right|\right|_{L^2}
\end{align}
the first term is exactly the same as in the previous sections and the other two go to zero as \(\delta\to 0\) as shown above.  
\end{proof}

\appendix
\section{Proofs of Consistency}
\label{sec: proofs of consistency}
\subsection{Case \(H^*(m)>1/2\)}
\begin{lemma}\label{lem: asymptotic variance m=1 H>1/2}
For \(m=1\) and \(H\in(1/2,1)\), the diagonal and off-diagonal contributions in \eqref{eq: second moment of chaos term} satisfy
\[
E_{2}^{\mathrm{Diag}} \sim \frac{\kappa_H^2}{H^2(2H-1)^2}\,\delta^{3+2(2H-2)}\varepsilon^{2(2-2H)},
\]
and
\[
E_{2}^{\mathrm{Off}} \sim
\begin{cases}
\kappa_H^2\,\zeta(2(2-2H))\,\delta^{3+2(2H-2)}\varepsilon^{2(2-2H)}, & H\in(1/2,3/4),\\[0.3em]
\kappa_H^2\,|\log(\delta)|\,\delta^{2}\varepsilon^{2(2-2H)}, & H=3/4,\\[0.3em]
\kappa_H^2\,a(H,1,0)\,\delta^{2}\varepsilon^{2(2-2H)}, & H\in(3/4,1).
\end{cases}
\]
\end{lemma}
\begin{proof}[of Lemma \ref{lem: asymptotic variance m=1 H>1/2}]
\begin{align}
E_{2}^{\text{Off}} &= \sum_{i,j=0, i\neq j}^{N-1}\delta^4\int_{[0,1]^4}\rho((u-v)\delta/\varepsilon+|i-j|\delta/\varepsilon)\rho((l-s)\delta/\varepsilon + |i-j|\delta/\varepsilon)dudvdlds\\
&= \delta^4\int_{[0,1]^4}\left(\sum_{k=1}^{N-1}(N-k)\rho((u-v)\delta/\varepsilon+k\delta/\varepsilon)\rho((l-s)\delta/\varepsilon+k\delta/\varepsilon)\right)dudvdlds\\
&\sim \kappa_H^2 \delta^4\int_{[0,1]^4}\sum_{k=1}^{N-1}(N-k)|k\delta/\varepsilon|^{2(2H-2)}dudvdlds\\
&=\kappa_H^2 \delta^3\varepsilon^{2(2-2H)}\sum_{k=1}^{N-1}(1-k/N)(k/N)^{2(2H-2)}\\
\end{align}
Now if \(H^*(m)=H\in(1/2,3/4)\) then \(2(2H-2)<-1\) and we may use the asymptotics in Lemma \ref{lem: sum asymptotics} to get
\begin{align}
&\sim \delta^{3+2(2H-2)}\varepsilon^{2(2-2H)}\zeta(2(2-2H))\kappa_H^2.\\
\end{align}
If \(H\in(3/4,1)\) then \(2(2H-2)>-1\) and we get 
\begin{align}
&\sim \delta^{2}\varepsilon^{2(2-2H)}\left(\int_0^1(1-x)x^{2(2H-2)}dx\right)\kappa_H^2\\
&= \delta^{2}\varepsilon^{2(2-2H)}a(H,1,0)\kappa_H^2
\end{align}
In the borderline case \(H=3/4\) we have that \(2(2H-2)=-1\) and we get again by Lemma \ref{lem: sum asymptotics} that
\begin{align}
&\sim \delta^{2}\varepsilon^{2(2-2H)}|\log(\delta)|\kappa_H^2
\end{align}
Now for the diagonal terms we have 
\begin{align}
E_{2}^{\text{Diag}} &= \sum_{i=0}^{N-1}\delta^4\int_{[0,1]^4}\rho((u-v)\delta/\varepsilon)\rho((l-s)\delta/\varepsilon)dudvdlds\\
&= \delta^3(\varepsilon/\delta)^4\left(\int_{[0,\delta/\varepsilon]^2}\rho(u-v)dudv\right)^2\\
&\sim \f{\kappa_H^2}{H^2(2H-1)^2}\delta^{3+2(2H-2)}\varepsilon^{2(2-2H)}
\end{align}
\end{proof}
\begin{lemma}\label{lem: asymptotic variance m>1 H>1/2}
For \(m>1\) and \(H>1/2\), and any \(1\leq r\leq m-1\), the diagonal and off-diagonal contributions in \eqref{eq: second moment of chaos term} satisfy
\[
E_{2m-2r}^{\mathrm{Diag}} \sim \kappa_H^{2m}\,\delta^{3+2m(2H-2)}\varepsilon^{2m(2-2H)},
\]
and, with \(r_I\) defined as in \eqref{eq: rI rS def},
\[
E_{2m-2r}^{\mathrm{Off}} \sim
\begin{cases}
\dfrac{8\,a(H,m,r)\,\kappa_H^2}{(r(2H-2)+2)^2(r(2H-2)+1)^2}\,\delta^{2r(2H-2)+2}\,\varepsilon^{2m(2-2H)}, & r>r_I,\\[0.6em]
\dfrac{2\kappa_H^2}{(r(2H-2)+2)^2(r(2H-2)+1)^2}\,\delta^{2m(2H-2)+3}\,|\log(\delta)|\,\varepsilon^{2m(2-2H)}, & r=r_I,\\[0.6em]
\dfrac{8\,\zeta(-2(m-r)(2H-2))\,\kappa_H^2}{(r(2H-2)+2)^2(r(2H-2)+1)^2}\,\delta^{2m(2H-2)+3}\,\varepsilon^{2m(2-2H)}, & r<r_I.
\end{cases}
\]
\end{lemma}
\begin{proof}[of Lemma \ref{lem: asymptotic variance m>1 H>1/2}]
\begin{align}
    \label{eq: sum off diagonals m>1 1}
E_{2m-2r}^{\text{Off}}=&\delta^4\int_{[0,1]^4}\rho((u-v)\delta/\varepsilon)^r\rho((l-s)\delta/\varepsilon)^r \left(\sum_{i,j=0, i\neq j}^{N-1}\rho((u-s)\delta/\varepsilon+|i-j|\delta/\varepsilon)^{m-r}\right.\\
&\left.\rho((v-l)\delta/\varepsilon+|i-j|\delta/\varepsilon)^{m-r}\right)dudvdlds\\
=& \delta^4\int_{[0,1]^4}\rho((u-v)\delta/\varepsilon)^r\rho((l-s)\delta/\varepsilon)^r2 \left(\sum_{k=1}^{N-1}(N-k)\rho((u-s)\delta/\varepsilon+k\delta/\varepsilon)^{m-r}\right. \\ 
&\left. \rho((v-l)\delta/\varepsilon+k\delta/\varepsilon)^{m-r}\right)dudvdlds
\end{align}
in the limit, \(u-s\) and \(v-l\) are negligible in front of \(k\) except for finitely many values of k. Therefore, for small enough \(\varepsilon\)
\begin{multline}
    \label{eq: sum off diagonals m>1 2}
\sim \delta^4\int_{[0,1]^4}\rho((u-v)\delta/\varepsilon)^r\rho((l-s)\delta/\varepsilon)^r2\left(\sum_{k=1}^{N-1}(N-k)|k\delta/\varepsilon|^{2(m-r)(2H-2)}\right)dudvdlds\\
\sim \delta^2\varepsilon^{2(m-r)(2-2H)}\int_{[0,1]^4}\rho((u-v)\delta/\varepsilon)^r\rho((l-s)\delta/\varepsilon)^r\\ 
2\left(\f{1}{N}\sum_{k=1}^{N-1}(1-k/N)|k/N|^{2(m-r)(2H-2)}\right)dudvdlds
\end{multline}
For \(r\geq m/2\) it is clear that \(2(m-r)(2H-2)>-1\) whenever \(H^*(m)>1/2\) and \(m>1\). Note that \(H^*(m)>1/2\) is equivalent to \(H>1-1/(2m)\) hence \(2(m-r)(2H-2)>m(2H-2)>-1\)
and the Riemann sum diverges in these cases. For the values \(r<m/2\) we cannot conclude in general; with \(r_I\) defined as in \eqref{eq: rI rS def}, the divergent regime is exactly \(r>r_I\). For these terms we get from Lemma \ref{lem: sum asymptotics} that
\begin{align}
&\sim 
\begin{multlined}[t]
2\delta^2\varepsilon^{2(m-r)(2-2H)}\int_{[0,1]^4}\rho((u-v)\delta/\varepsilon)^r\rho((l-s)\delta/\varepsilon)^r \\ 
\left(\int_0^1(1-x)x^{2(m-r)(2H-2)}\right)dudvdlds
\end{multlined}\\
&\sim 2a(H,m,r)\delta^2\varepsilon^{2(m-r)(2-2H)}\left(\int_{[0,1]^2}\rho((u-v)\delta/\varepsilon)^r dudv\right)^2\\
&= 2a(H,m,r)\delta^{-2}\varepsilon^{2(m-r)(2-2H)+4}\left(\int_{[0,\delta/\varepsilon]^2}\rho(u-v)^r dudv\right)^2
\end{align}
Clearly \(r(2H-2)<-1\) whenever \(1\leq r\leq m-1\) and \(H^*(m)>1/2\). Therefore we may use Lemma \ref{lem: integral asymptotics} to get the asymptotics
\begin{align}
&\sim \frac{8a(H,m,r)\kappa_H^2}{(r(2H-2)+2)^2(r(2H-2)+1)^2}\delta^{-2}\varepsilon^{2(m-r)(2-2H)+4}(\delta/\varepsilon)^{2r(2H-2)+4}\\
&=\frac{8a(H,m,r)\kappa_H^2}{(r(2H-2)+2)^2(r(2H-2)+1)^2}\delta^{2r(2H-2)+2}\varepsilon^{2m(2-2H)}
\end{align}
For the cases where \(r<r_I\) (may be none) we have that by using the appropriate case in Lemma \ref{lem: sum asymptotics} instead that the sum is asymptotically bounded 
by 
\begin{equation}
\sim \frac{8\zeta(-2(m-r)(2H-2))\kappa_H^2}{(r(2H-2)+2)^2(r(2H-2)+1)^2}\delta^{2m(2H-2)+3}\varepsilon^{2m(2-2H)} .
\end{equation}
If instead \(r=r_I\) meaning that \(2(m-r)(2H-2)=-1\) we get that the sum is asymptotically equivalent to 
\begin{equation}
\sim \frac{2 \kappa_H^2}{(r(2H-2)+2)^2(r(2H-2)+1)^2}\delta^{2m(2H-2)+3}|\log(\delta)|\varepsilon^{2m(2-2H)}
\end{equation}
The sum of diagonal terms may be written as 
\begin{align}
    \label{eq: sum of diagonal m>1}
&E_{2m-2r}^{\text{Diag}} \\ 
&= \delta^3 \int_{[0,1]^4}\rho((u-v)\delta/\varepsilon)^r\rho((l-s)\delta/\varepsilon)^r\rho((u-s)\delta/\varepsilon)^{m-r}\rho((v-l)\delta/\varepsilon)^{m-r}dudvdsdl\\
&\leq\delta^3\left(\int_{[0,1]^2}\rho((u-v)\delta/\varepsilon)^m dudv\right)^{2r/m}\left(\int_{[0,1]^2}\rho((u-v)\delta/\varepsilon)^m dudv\right)^{2(m-r)/m}\\
&\sim \kappa_H^{2m} \delta^3(\varepsilon/\delta)^4 (\delta/\varepsilon)^{2m(2H-2)+4}\\
&= \kappa_H^{2m} \delta^{3+2m(2H-2)}\varepsilon^{2m(2-2H)}
\end{align}
\end{proof}
\subsection{Case \(H^*(m)<1/2\)}
\begin{lemma}\label{lem: second chaos m=1 H<1/2}
For \(m=1\) and \(H<1/2\), the diagonal and off-diagonal contributions in \eqref{eq: second moment of chaos term} satisfy
\[
E_{2}^{\mathrm{Off}} \sim \kappa_H^2\,\zeta(2(2-2H))\,\delta^{3+2(2H-2)}\varepsilon^{2(2-2H)},\qquad
E_{2}^{\mathrm{Diag}} \sim 4\Bigl(\int_0^{\infty}\rho(s)\,ds\Bigr)^2 \delta\,\varepsilon^{2}.
\]
\end{lemma}
\begin{proof}[of Lemma \ref{lem: second chaos m=1 H<1/2}]
\begin{align}
E_{2}^{\text{Off}} &=\sum_{i,j=0, i\neq j}^{N-1}\left\langle\ g_{i}^{1,0,\varepsilon},g_{j}^{1,0,\varepsilon}\right\rangle_{L^2(\mathbb{R}^{2})}\\
&= \sum_{i,j=0, i\neq j}^{N-1}\delta^4\int_{[0,1]^4}\rho((u-v)\delta/\varepsilon+|i-j|\delta/\varepsilon)\rho((l-s)\delta/\varepsilon + |i-j|\delta/\varepsilon)dudvdlds\\
&= \delta^4\int_{[0,1]^4}\left(\sum_{k=1}^{N-1}(N-k)\rho((u-v)\delta/\varepsilon+k\delta/\varepsilon)\rho((l-s)\delta/\varepsilon+k\delta/\varepsilon)\right)dudvdlds\\
&\sim \kappa_H^2 \delta^4\int_{[0,1]^4}\sum_{k=1}^{N-1}(N-k)|k\delta/\varepsilon|^{2(2H-2)}dudvdlds\\
&=\kappa_H^2 \delta^3\varepsilon^{2(2-2H)}\sum_{k=1}^{N-1}(1-k/N)(k/N)^{2(2H-2)}\\
\end{align}
Now if \(H^*(m)<1/2\) then \(2(2H-2)<-2\) and we may use the asymptotics in Lemma \ref{lem: sum asymptotics} to get
\begin{align}
&\sim \delta^{3+2(2H-2)}\varepsilon^{2(2-2H)}\zeta(2(2-2H))\kappa_H^2\\
\end{align}
Now for the diagonal terms we have 
\begin{align}
E_{2}^{\text{Diag}} &=\sum_{i=0}^{N-1}\left\langle\ g_{i}^{1,0,\varepsilon},g_{i}^{1,0,\varepsilon}\right\rangle_{L^2(\mathbb{R}^{2})}\\
&= \sum_{i=0}^{N-1}\delta^4\int_{[0,1]^4}\rho((u-v)\delta/\varepsilon)\rho((l-s)\delta/\varepsilon)dudvdlds\\
&= \delta^3(\varepsilon/\delta)^4\left(\int_{[0,\delta/\varepsilon]^2}\rho(u-v)dudv\right)^2\\
&\sim 4\left(\int_{0}^{\infty}\rho(s)ds\right)^2 \delta \varepsilon^{2}
\end{align}
\end{proof}
\begin{lemma}\label{lem: asymptotics m>1 H<1/2}
For \(m>1\) and \(H^*(m)<1/2\), and any \(1\leq r\leq m-1\), with \(r_I, r_S\) as in \eqref{eq: rI rS def}, the diagonal and off-diagonal contributions in \eqref{eq: second moment of chaos term} satisfy
\[
E_{2m-2r}^{\mathrm{Diag}} \lesssim 4\Bigl(\int_0^{\infty}\rho(s)^m\,ds\Bigr)^2 \delta\,\varepsilon^{2},
\]
and
\[
E_{2m-2r}^{\mathrm{Off}} \lesssim
\begin{cases}
\delta^{3}(\varepsilon/\delta)^{2m(2-2H)}, & r< r_I\wedge r_S,\\[0.3em]
\varepsilon^{2(m-r)(2-2H)+2}, & r> r_I\vee r_S,\\[0.3em]
\delta^{2+2(m-r)(2H-2)}(\varepsilon/\delta)^{2m(2-2H)}, & r_S< r < r_I,\\[0.3em]
\delta\,\varepsilon^{2}(\varepsilon/\delta)^{2(m-r)(2-2H)}, & r_I< r < r_S,
\end{cases}
\]
with implicit constants depending only on \(H, m\) and \(r\).
\end{lemma}
\begin{proof}[of Lemma \ref{lem: asymptotics m>1 H<1/2}]
Just as in \eqref{eq: sum of diagonal m>1} we have that
\begin{equation}
E_{2m-2r}^{\text{Diag}}\leq \delta^3\left(\int_{[0,1]^2}\rho((u-v)\delta/\varepsilon)^m dudv\right)^2
\end{equation}
and so, in the case \(H^*(m)<1/2\) we have that via Lemma \ref{lem: integral asymptotics}
\begin{equation}
E_{2m-2r}^{\text{Diag}}\lesssim 4\left(\int_0^{\infty}\rho(s)^m\right)^2\delta\varepsilon^2
\end{equation}
For the off-diagonal terms again as in \eqref{eq: sum off diagonals m>1 1} and \eqref{eq: sum off diagonals m>1 2} that 
\begin{align}
& \begin{multlined}
E_{2m-2r}^{\text{Off}} \sim \delta^2\varepsilon^{2(m-r)(2-2H)}\int_{[0,1]^4}\rho((u-v)\delta/\varepsilon)^r\rho((l-s)\delta/\varepsilon)^r\\ 
2\left(\f{1}{N}\sum_{k=1}^{N-1}(1-k/N)|k/N|^{2(m-r)(2H-2)}\right)dudvdlds
\end{multlined}\\
& \begin{multlined}=\delta^2\varepsilon^{2(m-r)(2-2H)}2\left(\f{1}{N}\sum_{k=1}^{N-1}(1-k/N)|k/N|^{2(m-r)(2H-2)}\right) \\ (\varepsilon/\delta)^4\left(\int_{[0,\delta/\varepsilon]^2}\rho((u-v))^r dudv\right)^2
\end{multlined}
\end{align}
each of the cases follow from choosing the combination of asymptotics from Lemmas \ref{lem: integral asymptotics} and \ref{lem: sum asymptotics} that 
apply to each case: 
\begin{enumerate}
    \item If \(r\leq r_S\wedge r_I\) then both the integral and sum are divergent and we get the asymptotics \begin{equation}
        \sim \delta^2 \varepsilon^{2(m-r)(2-2H)}\delta^{2(m-r)(2H-2)+1}(\varepsilon/\delta)^{2r(2-2H)} = \delta^3 (\varepsilon/\delta)^{2m(2-2H)}
    \end{equation}
    \item If \(r\geq r_S\vee r_I\) then both the integral and sum are convergent and we get the asymptotics \begin{equation}
        \sim \delta^2 \varepsilon^{2(m-r)(2-2H)}(\varepsilon/\delta)^{2} = \varepsilon^{2(m-r)(2-2H)+2}
    \end{equation}
    \item If \(r_S\leq r \leq r_I\) then the sum is convergent but the integral is divergent and we get the asymptotics \begin{equation}
        \sim \delta^2 \varepsilon^{2(m-r)(2-2H)}(\varepsilon/\delta)^{2r(2-2H)} = \delta^{2+2(m-r)(2H-2)}(\varepsilon/\delta)^{2m(2-2H)}
    \end{equation}
    \item If \(r_I\leq r \leq r_S\) then the sum is divergent but the integral is convergent and we get the asymptotics \begin{equation}
        \sim \delta^2 \varepsilon^{2(m-r)(2-2H)}\delta^{2(m-r)(2H-2)+1}(\varepsilon/\delta)^{2} = \delta \varepsilon^2 (\varepsilon/\delta)^{2(m-r)(2-2H)}
    \end{equation}
\end{enumerate}
\end{proof}
\section{Proof of Lemma \ref{lem: convergence of the inner product}}
\label{sec: convergence inner product appendix}
\begin{proof}[of Lemma \ref{lem: convergence of the inner product}]
Note first that 
\begin{align}
&\left\langle f_{\varepsilon,\delta}, \frac{2}{K(2H-1,2)}f\right\rangle_{L^2(\mathbb{R}^2)} = \\
&\begin{multlined}
N^{2H^*(m)-1 + (2-2H)}\varepsilon^{2(H^*(m)-1)}\sum_{k=0}^{N-1}\int_{[k\delta,(k+1)\delta]^2}\int_0^1 \left(\int_{-\infty}^{u\vee s}\tau_uK^{\varepsilon}(x)(s-x)_{+}^{H-3/2}dx\right) \\ \left(\int_{-\infty}^{v\vee s}\tau_v K^{\varepsilon}(y)(s-y)_{+}^{H-3/2}dy\right)\rho^{\varepsilon}(u-v)^{m-1}dsdudv
\end{multlined}
\end{align}
We will now single out the small domains where the fractional OU kernel does not decay like \(s^{H-3/2}\) as when the integration domain passes through them the asymptotic behaviour of the kernel cannot be used even for small enough \(\varepsilon\) and we need 
to control the contribution of these small domains to the limit. We split the inner integral over \([0,1]\) into three parts as follows
\begin{align}
&\begin{multlined}
\sum_{k=0}^{N-1}\int_{[k\delta,(k+1)\delta]^2}\int_0^{u\wedge v}\left(\int_{-\infty}^{u\vee s}\tau_uK^{\varepsilon}(x)(s-x)_{+}^{H-3/2}dx\right) \\ \left(\int_{-\infty}^{v\vee s}\tau_v K^{\varepsilon}(y)(s-y)_{+}^{H-3/2}dy\right)\rho^{\varepsilon}(u-v)^{m-1}dsdudv 
\end{multlined}\\
&\begin{multlined}
    + \sum_{k=0}^{N-1}\int_{[k\delta,(k+1)\delta]^2}\int_{u\wedge v}^{u\vee v}\left(\int_{-\infty}^{u\vee s}\tau_uK^{\varepsilon}(x)(s-x)_{+}^{H-3/2}dx\right) \\ \left(\int_{-\infty}^{v\vee s}\tau_v K^{\varepsilon}(y)(s-y)_{+}^{H-3/2}dy\right)\rho^{\varepsilon}(u-v)^{m-1}dsdudv
\end{multlined}\\
&\begin{multlined}
+ \sum_{k=0}^{N-1}\int_{[k\delta,(k+1)\delta]^2}\int_{u\vee v}^1\left(\int_{-\infty}^{u\vee s}\tau_uK^{\varepsilon}(x)(s-x)_{+}^{H-3/2}dx\right) \\ \left(\int_{-\infty}^{v\vee s}\tau_v K^{\varepsilon}(y)(s-y)_{+}^{H-3/2}dy\right)\rho^{\varepsilon}(u-v)^{m-1}dsdudv
\end{multlined} 
\end{align}
and denote them by (1), (2) and (3) respectively. We now filter the terms of order \(\mathcal{O}(1)\) in each term. Note that in (1) the aforementioned small domain in the fOU kernels does not appear as the integration domain is always away from the singularity, in (2) it appears once and in (3) it appears twice.
We begin with the third one. We first make the split as in Lemma \ref{lem: fOU kernel asymptotics}  
\begin{equation}
\int_{-\infty}^{u}\tau_uK^{\varepsilon}(x)(s-x)_{+}^{H-3/2}dx = A^{<}_u + A^{\varepsilon}_u
\end{equation}
and the same for the \(y\) variable leading to the decomposition of the integrand into four terms. 
\begin{equation}
(3) = \sum_{k=0}^{N-1}\int_{[k\delta,(k+1)\delta]^2}\int_{u\vee v}^1 (A^{<}_u A^{<}_v + A^{<}_u A^{\varepsilon}_v + A^{\varepsilon}_u A^{<}_v + A^{\varepsilon}_u A^{\varepsilon}_v) \rho^{\varepsilon}(u-v)^{m-1}dsdudv
\end{equation}
which we denote (3.1), (3.2), (3.3) and (3.4) respectively. We will show that the only term that contributes to the limit is (3.1) and the rest vanish after renormalisation. We begin with (3.4). Using the asymptotics in 
Lemma \ref{lem: fOU kernel asymptotics}
\begin{align}
(3.4)&\leq \sum_{k=0}^{N-1}\int_{[k\delta,(k+1)\delta]^2}\int_{u\vee v}^1 \varepsilon (s-u)^{H-3/2}(s-v)^{H-3/2}\rho^{\varepsilon}(u-v)^{m-1}dsdudv\\
&\leq \sum_{k=0}^{N-1}\int_{[k\delta,(k+1)\delta]^2}\beta(H-1/2,2-2H)\varepsilon |u-v|^{2H-2}\rho^{\varepsilon}(u-v)^{m-1}dudv\\
&=\beta(H-1/2,2-2H)\varepsilon\varepsilon^{(2H-2)}\delta^{-1}\int_{[0,\delta]^2}\left|\frac{u-v}{\varepsilon}\right|^{2H-2}\rho\left(\frac{u-v}{\varepsilon}\right)^{m-1}dudv\\
&=\beta(H-1/2,2-2H)\varepsilon\varepsilon^{(2H-2)}\delta^{-1}\varepsilon^2\int_{[0,\delta/\varepsilon]^2}|u-v|^{2H-2}\rho(u-v)^{m-1}dudv\\
&\sim C \varepsilon\varepsilon^{(2H-2)}\delta^{-1}\varepsilon^2 (\delta/\varepsilon)^{m(2H-2)+2} = C\varepsilon^{(m-1)(2H-2)+1}\delta^{2H^*(m)-1}
\end{align}
where in the second inequality we use the bound in Lemma \ref{lem: aux integral 2}. Thus after re-normalisation
\begin{equation}
N^{1-2H^*(m)+2H-2}\varepsilon^{2(1-H^*(m))}(3.4) \lesssim \varepsilon^{2H-1}\delta^{2H-2} = (\varepsilon/\delta)^{2H-2}\delta^{4H-3}
\end{equation}
which vanishes as \(\varepsilon/\delta\to 0\) and \(H>3/4\). Now in the term (3.3) the constant behaviour appears only once 
\begin{align}
(3.3)&\leq \sum_{k=0}^{N-1}\int_{k\delta}^{(k+1)\delta}\int_{k\delta}^{(k+1)\delta}\varepsilon^{3/2-H}(s-u)^{H-3/2}(s-v)^{2H-2}\rho^{\varepsilon}(u-v)^{m-1}dsdudv\\
\end{align}
In this case we need to split into two cases. If \(3/4<H<5/6\) we can use the asymptotics Lemma \ref{lem: aux integral 3} to get the bound 
\begin{align}
(3.3) &\leq \sum_{k=0}^{N-1}\int_{k\delta}^{(k+1)\delta}\int_{k\delta}^{(k+1)\delta}\varepsilon^{3/2-H}|u-v|^{H-3/2}\rho^{\varepsilon}(u-v)^{m-1}dudv\\
&=\varepsilon^{3/2-H}\varepsilon^{3H-7/2}\delta^{-1}\int_{[0,\delta]^2}\left|\frac{u-v}{\varepsilon}\right|^{3H-5/2}\rho\left(\frac{u-v}{\varepsilon}\right)dudv\\
&\lesssim \varepsilon^{3/2-H}\delta^{-1}\varepsilon^{2H-2 + H-1/2}\varepsilon^2 (\delta/\varepsilon)^{m(2H-2)+(H-1/2)+2}\\
&= \varepsilon^{3/2-H + (m-1)(2-2H)}\delta^{2H^*(m)-1 + (H-1/2)}
\end{align}
So after renormalisation the order of this term is controlled by 
\begin{equation}
N^{1-2H^*(m)+2H-2}\varepsilon^{2(1-H^*(m))}(3.3)\lesssim \varepsilon^{H-1/2}\delta^{3H-5/2}.
\end{equation}
To make sure this vanishes we would in principle have to impose the subsampling condition \(\delta = \varepsilon^{\alpha}\) for \(\alpha<\frac{H-1/2}{5/2-3H}\) for this term to vanish. However, for \(3/4<H<5/6\) we can bound \(1<\frac{H-1/2}{5/2-3H}\) 
and it does not add an extra subsampling condition rather than the usual regime. Note that the term (3.2) is controlled equally due to the interchangeability of the roles of \(u\) and \(v\) and the fact that changing the singularity within the integral 
does not alter the order but just the constant in front as seen in Lemma \ref{lem: aux integral 3}. 

If instead \(5/6\leq H<1\) we cannot use these bounds as the beta functions are no longer defined as the integrals diverge and we need to find the order explicitly. The strategy employed is to localize each of the variables into squares of the same size to get good control both on and off the diagonal 
and single out the singularities. That means to write 
\begin{align}
(3.3)& = \sum_{k=0}^{N-1}\int_{k\delta}^{(k+1)\delta}\int_{k\delta}^{(k+1)\delta}\int_0^1\varepsilon^{3/2-H}(s-u)_{+}^{H-3/2}(s-v)_{+}^{2H-2}ds \rho^{\varepsilon}(u-v)^{m-1}dudv\\
&=\sum_{i,j=0}^{N-1}\int_{i\delta}^{(i+1)\delta}\int_{i\delta}^{(i+1)\delta}\int_{j\delta}^{(j+1)\delta}\varepsilon^{3/2-H}(s-u)_{+}^{H-3/2}(s-v)_{+}^{2H-2}ds \rho^{\varepsilon}(u-v)^{m-1}dudv\\
&\begin{multlined}
=\varepsilon^{3/2-H}\varepsilon^{H-3/2}\varepsilon^{2H-2}\sum_{i,j=0}^{N-1}\int_{[0,\delta]^3}\left(\frac{s-u}{\varepsilon}+\frac{(j-i)\delta}{\varepsilon}\right)_{+}^{H-3/2} \\ \left(\frac{s-v}{\varepsilon}+\frac{j-i}{\varepsilon}\right)_{+}^{2H-2}\rho\left(\frac{u-v}{\varepsilon}\right)^{m-1}dsdudv
\end{multlined}
\end{align}
Here we need to study separately the off and on diagonal terms. For the off-diagonal ones we have that as in previous calculations the behaviour within the two first term in the integrals is dominated 
by the distance between the squares rather than the particular values within the square, so we can write 
\begin{align}
&\begin{multlined}
=\varepsilon^{3/2-H}\varepsilon^{H-3/2}\varepsilon^{2H-2}\sum_{i,j=0}^{N-1}\int_{[0,\delta]^3}\left(\frac{s-u}{\varepsilon}+\frac{(j-i)\delta}{\varepsilon}\right)_{+}^{H-3/2} \\ \left(\frac{s-v}{\varepsilon}+\frac{j-i}{\varepsilon}\right)_{+}^{2H-2}\rho\left(\frac{u-v}{\varepsilon}\right)^{m-1}dsdudv
\end{multlined}\\
&=\varepsilon^{3/2-H}\varepsilon^{H-3/2}\varepsilon^{2H-2}\left(\delta\sum_{i,j=0}^{N-1}\left(\frac{(j-i)\delta}{\varepsilon}\right)_{+}^{3H-7/2}\right)\int_{[0,\delta]^2}\rho\left(\frac{u-v}{\varepsilon}\right)^{m-1}dudv\\
&\sim \varepsilon^{3/2-H}\varepsilon^{(m-1)(2-2H)}\delta^{(m-1)(2H-2)+1}(\delta\sum_{k=1}^{N-1}(1-\delta-k\delta)(k\delta)^{3H-7/2})\\ 
&\sim \varepsilon^{3/2-H}\varepsilon^{(m-1)(2-2H)}\delta^{(m-1)(2H-2)+1}\int_0^1(1-x)x^{3H-7/2}dx
\end{align}
Note that the latter integral is finite iff \(3H-7/2>-1\) which is the case for \(H>5/6\). After renormalization we thus get that 
\begin{align}
N^{1-2H^*(m)+2H-2}\varepsilon^{2(1-H^*(m))}(3.3)_{\text{off}} &\lesssim \varepsilon^{H-1/2} 
\end{align}
We now study the diagonal terms for which we take advantage of the asymptotics in Lemma \ref{lem: triple integral asymptotics}.
\begin{align}
&\varepsilon^{3/2-H}\varepsilon^{H-3/2}\varepsilon^{2H-2}\delta^{-1}\int_{[0,\delta]^3}\left(\frac{s-u}{\varepsilon}\right)_{+}^{H-3/2}\left(\frac{s-v}{\varepsilon}\right)_{+}^{2H-2}\rho\left(\frac{u-v}{\varepsilon}\right)^{m-1}dsdudv\\
&=\varepsilon^{3/2-H}\varepsilon^{H-3/2}\varepsilon^{2H-2}\delta^{-1}\varepsilon^3\int_{[0,\delta/\varepsilon]^3}(s-u)_{+}^{H-3/2}(s-v)_{+}^{2H-2}\rho(u-v)^{m-1}dsdudv\\
&\lesssim \varepsilon^{3/2-H}\varepsilon^{H-3/2}\varepsilon^{2H-2}\delta^{-1}\varepsilon^3(\delta/\varepsilon)^{3H-7/2+(m-1)(2H-2)+3}\\
&=\varepsilon^{3/2-H - 2(H^*(m)-1)+2H-2}\delta^{2H^*(m)-1+2-2H+3H-7/2 +1}
\end{align}
and after rescaling we get that 
\begin{align}
N^{1-2H^*(m)+2H-2}\varepsilon^{2(1-H^*(m))}(3.3)_{\text{diag}} &\lesssim (\varepsilon/\delta)^{H-1/2}\delta^{4H-3}
\end{align}
which vanishes as \(\varepsilon/\delta\to 0\) and \(H>3/4\). This concludes the study of the term (3.3) due to the interchangeability of \(u\) and \(v\). 

The term (3.2) is controlled equally due to the interchangeability of the roles of \(u\) and \(v\). We now study the term (3.1) which is the only one that contributes 
\begin{align}
&\begin{multlined}
    (3.1) = \sum_{k=0}^{N-1}\int_{[k\delta,(k+1)\delta]^2}\int_{u\vee v}^1 \left(\int_{-\infty}^{u-\varepsilon}\tau_u K^{\varepsilon}(x)(s-x)_{+}^{H-3/2}dx\right) \\ \left(\int_{-\infty}^{u-\varepsilon}\tau_v K^{\varepsilon}(y)(s-y)_{+}^{H-3/2}dy\right)\rho^{\varepsilon}(u-v)^{m-1}dudv 
\end{multlined}\\
&\begin{multlined}
\sim p^2\beta(H-1/2,2-2H)^2 \sum_{k=0}^{N-1}\int_{[k\delta,(k+1)\delta]^2}\varepsilon^{2-2H} \\ \left(\int_{u\vee v}^1 (s-u)^{2H-2}(s-v)^{2H-2}ds\right)\rho^{\varepsilon}(u-v)^{m-1}dudv
\end{multlined}
\end{align}
which is the term of order \(\mathcal{O}(1)\) which we will later put together with the ones coming from the other terms. We study now the term (2) for which we apply a similar split. In this case we split as before the term in which the constant behaviour of the fOU kernel appears and get 
\begin{align}
&\begin{multlined}
(2) =\sum_{k=0}^{N-1}\int_{[k\delta,(k+1)\delta]^2}\int_{v}^{u}\left(\int_{-\infty}^{s}\tau_uK^{\varepsilon}(x)(s-x)_{+}^{H-3/2}dx\right) \\ \left(\int_{-\infty}^{v}\tau_v K^{\varepsilon}(y)(s-y)_{+}^{H-3/2}dy\right)\rho^{\varepsilon}(u-v)^{m-1}dsdudv
\end{multlined}\\
& = \sum_{k=0}^{N-1}\int_{[k\delta,(k+1)\delta]^2}\int_{v}^{u}\left(\int_{-\infty}^{s}\tau_uK^{\varepsilon}(x)(s-x)_{+}^{H-3/2}dx\right)\left(A^{<}_v + A^{\varepsilon}_v\right)\rho^{\varepsilon}(u-v)^{m-1}dsdudv\\
& = (2.1) + (2.2)
\end{align}
We will show just as before that the term (2.2) vanishes and the term (2.1) will contribute to the limit. By the analysis in the fOU kernel of Lemma \ref{lem: fOU kernel asymptotics} we can bound 
\begin{align}
(2.2)&\lesssim \sum_{k=0}^{N-1}\int_{[k\delta,(k+1)\delta]}\left(\int_v^u \varepsilon^{3/2-H}(s-v)^{H-3/2}(u-s)^{2H-2}ds\right)\rho^{\varepsilon}(u-v)^{m-1}dudv\\
&\sim \sum_{k=0}^{N-1}\int_{[k\delta,(k+1)\delta]^2}\varepsilon^{3/2-H}(u-v)^{3H-5/2}\rho^{\varepsilon}(u-v)^{m-1}dudv\\
&=\varepsilon^{3/2-H}\delta^{-1}\int_{[0,\delta]^2}(u-v)^{3H-5/2}\rho^{\varepsilon}(u-v)^{m-1}dudv\\
&=\varepsilon^{3/2-H}\delta^{-1}\varepsilon^{3H-5/2}\int_{[0,\delta]^2}\left(\frac{u-v}{\varepsilon}\right)^{3H-5/2}\rho\left(\frac{u-v}{\varepsilon}\right)^{m-1}dudv\\
&=\varepsilon^{3/2-H}\delta^{-1}\varepsilon^{3H-5/2}\varepsilon^2\int_{[0,\delta/\varepsilon]^2}(u-v)^{3H-5/2}\rho(u-v)^{m-1}dudv\\
&\lesssim \varepsilon^{(3/2-H) + (m-1)(2-2H)}\delta^{m(2H-2)+ H-1/2 + 1}\\ 
&=\varepsilon^{(3/2-H) + (m-1)(2-2H)}\delta^{2H^*(m)-1 + H-1/2}
\end{align}
and therefore after renormalisation we get that 
\begin{equation}
N^{1-2H^*(m)+2H-2}\varepsilon^{2(1-H^*(m))}(2.2) \lesssim \varepsilon^{H-1/2}\delta^{3H-5/2}
\end{equation}
which vanishes if \(5/6<H<1\) and if \(3/4<H<5/6\) we need to impose the subsampling condition \(\delta = \varepsilon^{\alpha}\) for \(\alpha<\frac{H-1/2}{5/2-3H}\) which again does not add any extra condition. We now study the term (2.1) which we simply write 
\begin{multline}
(2.1) \sim p^2\beta(H-1/2,2-2H)^2 \sum_{k=0}^{N-1}\int_{[k\delta,(k+1)\delta]^2}\varepsilon^{2-2H} \\ 
\left(\int_{v\wedge u}^{v\vee u} (s-u\wedge v)^{2H-2}(u\vee v -s)^{2H-2}ds\right) \rho^{\varepsilon}(u-v)^{m-1}dudv
\end{multline}

Finally we similarly write term (1) as 
\begin{multline}
(1)\sim p^2\beta(H-1/2,2-2H)^2 \sum_{k=0}^{N-1}\int_{[k\delta,(k+1)\delta]^2}\varepsilon^{2-2H} \\ 
\int_0^{u\wedge v} (u-s)^{2H-2}(v-s)^{2H-2}ds \rho^{\varepsilon}(u-v)^{m-1}dudv
\end{multline}

After filtering out the vanishing terms and just keeping the ones that contribute to the limit we get that 
\begin{align}
\left\langle f_{\varepsilon,\delta}, \frac{2}{K(2H-1,2)}f\right\rangle_{L^2(\mathbb{R}^2)} \\
&\begin{multlined}
\sim p^2\beta(H-1/2,2-2H)^2 \sum_{k=0}^{N-1}\int_{[k\delta,(k+1)\delta]^2}\varepsilon^{2-2H} \\ \int_0^1|s-u|^{2H-2}|s-v|^{2H-2}\rho^{\varepsilon}(u-v)^{m-1}dsdudv
\end{multlined}
\end{align}
We employ the same strategy as before to split into squares the inner integral and single out the diagonal singularity to show that it is indeed the off-diagonal behaviour 
that dominates. We write 
\begin{align}
&\begin{multlined}
\sim p^2\beta(H-1/2,2-2H)^2 \sum_{k=0}^{N-1}\int_{[k\delta,(k+1)\delta]^2}\varepsilon^{2-2H} \\ \int_0^1|s-u|^{2H-2}|s-v|^{2H-2}\rho^{\varepsilon}(u-v)^{m-1}dsdudv
\end{multlined}\\
&\begin{multlined}
= p^2\beta(H-1/2,2-2H)^2 \sum_{i,j=0}^{N-1}\int_{i\delta}^{(i+1)\delta}\int_{i\delta}^{(i+1)\delta}\int_{j\delta}^{(j+1)\delta}\varepsilon^{2-2H} \\ |s-u|^{2H-2}|s-v|^{2H-2}\rho^{\varepsilon}(u-v)^{m-1}dsdudv
\end{multlined}
\end{align}
We now study the off-diagonal terms 
\begin{align}
&\begin{multlined}
p^2\beta(H-1/2,2-2H)^2\sum_{i\neq j=0}^{N-1}\int_{[0,\delta]}\varepsilon^{2-2H}\left(|s-u| + |i-j|\delta\right)^{2H-2} \\ \left(|s-v| + |i-j|\delta\right)^{2H-2}\rho^{\varepsilon}(u-v)^{m-1}dsdudv
\end{multlined}\\
& \begin{multlined}
= p^2\beta(H-1/2,2-2H)^2\varepsilon^{2-2H}\varepsilon^{2(2H-2)}\sum_{i\neq j=0}^{N-1}\int_{[0,\delta]^3}\left(\frac{|s-u|}{\varepsilon} + |i-j|\delta/\varepsilon \right)^{2H-2} \\ \left( \frac{|s-v|}{\varepsilon} + |i-j|\delta/\varepsilon \right)^{2H-2}\rho\left(\frac{u-v}{\varepsilon}\right)^{m-1}dsdudv
\end{multlined}\\
& \begin{multlined}
\sim p^2\beta(H-1/2,2-2H)^2\varepsilon^{2-2H}\varepsilon^{2(2H-2)}\left(\delta\sum_{i\neq j=0}^{N-1}\left(|i-j|\delta/\varepsilon \right)^{2(2H-2)}\right) \\ \int_{[0,\delta]^2}\rho\left(\frac{u-v}{\varepsilon}\right)^{m-1}dsdudv
\end{multlined}\\
&\begin{multlined}
\sim p^2\beta(H-1/2,2-2H)^2\frac{2(1/\Gamma(2H-1))^{m-1}}{((m-1)(2H-2)+2)((m-1)(2H-2)+1)} \\ 
\left(\int_{[0,1]^2}|x-y|^{2(2H-2)}dxdy\right)\varepsilon^{2-2H}\varepsilon^{(m-1)(2-2H)}\delta^{(m-1)(2H-2)+1} 
\end{multlined}
\end{align}
after renormalization we get that the off-diagonal terms are asymptotically equivalent to 
\begin{align}
&\begin{multlined}
p^2\beta(H-1/2,2-2H)^2\frac{2(1/\Gamma(2H-1))^{m-1}}{((m-1)(2H-2)+2)((m-1)(2H-2)+1)} \\\left(\int_{[0,1]^2}|x-y|^{2(2H-2)}dxdy\right)
\end{multlined}\\
&= \beta(H-1/2,2-2H) \frac{2(1/\Gamma(2H-1))^m}{((m-1)(2H-2)+2)((m-1)(2H-2)+1)}\frac{1}{(4H-3)(2H-1)}.
\end{align}
We now show that the sum of diagonal terms vanishes. Indeed we can control this by 
\begin{align}
&\sum_{k=0}^{N-1}\int_{[0,\delta]^3}\varepsilon^{2-2H}|s-u|^{2H-2}|s-v|^{2H-2}\rho^{\varepsilon}(u-v)^{m-1}dsdudv \\
=&\delta^{-1}\varepsilon^{2-2H}\varepsilon^{2(2H-2)}\int_{[0,\delta]^3}\left|\frac{s-u}{\varepsilon}\right|^{2H-2}\left|\frac{s-v}{\varepsilon}\right|^{2H-2}\rho\left(\frac{u-v}{\varepsilon}\right)^{m-1}dsdudv\\
=&\delta^{-1}\varepsilon^{2-2H}\varepsilon^{2(2H-2)}\varepsilon^3\int_{[0,\delta/\varepsilon]^3}|s-u|^{2H-2}|s-v|^{2H-2}\rho(u-v)^{m-1}dsdudv\\
=&\delta^{-1}\varepsilon^{2-2H}\varepsilon^{2(2H-2)}\varepsilon^3 4\int_{[0,\delta/\varepsilon]^3}(s-u)_{+}^{2H-2}(s-v)_{+}^{2H-2}\rho(u-v)^{m-1}dsdudv\\
\asymp &\delta^{-1}\varepsilon^{2-2H}\varepsilon^{2(2H-2)}\varepsilon^3 (\delta/\varepsilon)^{2(2H-2)+(m-1)(2H-2)+3}\\ 
=&\varepsilon^{m(2-2H)}\delta^{(m-1)(2H-2)+2(2H-2)+2}
\end{align}
Therefore we have that the asymptotic order of the sum of diagonal terms after renormalisation is 
\begin{equation}
\delta^{2(2H-2)+1}
\end{equation}
and note that \(2(2H-2)+1>0\) if and only if \(H>3/4\) which is precisely the case. The proof is now concluded by noting that 
\begin{align}
\langle f_{\varepsilon,\delta}, f_1\rangle_{L^2(\mathbb{R}^2)}&\sim \\ 
&\begin{multlined}
\beta(H-1/2,2-2H) \frac{2(1/\Gamma(2H-1))^m}{((m-1)(2H-2)+2)((m-1)(2H-2)+1)} \\ \frac{1}{(4H-3)(2H-1)}\frac{K(2H-1,2)}{2} = R
\end{multlined}
\end{align}
\end{proof}
\section{Technical Lemmas}
\begin{lemma}
    \label{lem: integral asymptotics}
Let \(\beta(s)\) be a continuous non-increasing function satisfying \(\beta(0) = 1\) and \(|\beta(s)|\sim Ks^{\alpha}\) for some positive constant \(K\) and some \(\alpha<0\) as \(s\to\infty\). Then for any positive integer \(k\), as \(h\to\infty\):
\begin{enumerate}
    \item If \(\beta^k\) is not integrable along the positive half-axis (\(\alpha k>-1\)): \begin{equation}
        \label{eq: asymptotic if non-integrable}
        \int_{[a,a+h]^2}\beta(s-t)^k dsdt \sim \f{2K^k h^{\alpha k+2}}{(\alpha k+2)(\alpha k + 1)}
    \end{equation}
    \item If \(\alpha k<-1\):  \begin{equation}
        \label{eq: asymptotic if integrable}
        \int_{[a,a+h]^2}\beta(s-t)^k dsdt \sim 2h\int_0^{\infty}\beta(s)^k ds
    \end{equation}
    \item If \(\alpha k =-1\): \begin{equation}
        \label{eq: asymptotic if critical}
        \int_{[a,a+h]^2}\beta(s-t)^k dsdt \sim 2Kh\log(h)
    \end{equation}
\end{enumerate}
\end{lemma}
\begin{proof}
Taking the change of variables \(u = s-t\) and \(v=t\) we get  
\begin{equation}
\label{eq: change of variables}
\int_{[a,a+h]^2}\beta(s-t)^k dsdt = 2\int_0^h(h-s)\beta(s)^kds
\end{equation}
\textbf{Case 1}: 
We treat each of the terms arising in \eqref{eq: change of variables} separately 
for the first term we have that, for \(h>1\): 
\begin{align}
2h\int_0^h\beta(s)^kds &= 2h\int_0^1\beta(s)^kds + 2h\int_0^h \beta(s)^kds \\
&\sim \f{2K^k h^{\alpha k+1}}{\alpha k + 1} + 2Ch\\
&\sim \f{2K^k h^{\alpha k+1}}{\alpha k + 1}
\end{align}
where \(C\) is some constant we are not specific about since the first term dominates as \(h\to\infty\) whenever \(\alpha k>-1\). A similar technique on the second term leads to  
\begin{align}
2\int_0^hs\beta(s)^kds &\sim \f{2K^k h^{\alpha k+2}}{\alpha k+2} + 2C\\
&\sim \f{2K^k h^{\alpha k+2}}{\alpha k+2}
\end{align}
Now putting the two together through \eqref{eq: change of variables} we get 
\begin{align}
\int_{[a,a+h]^2}\beta(s-t)^k dsdt &\sim \f{2K^k h^{\alpha k+1}}{\alpha k + 1} - \f{2K^k h^{\alpha k+2}}{\alpha k+2}\\
&= \f{2K^k h^{\alpha k+2}}{(\alpha k+2)(\alpha k + 1)}
\end{align}
\textbf{Case 2:} in this case we can write \eqref{eq: change of variables} as
\begin{equation}
2h\int_0^{\infty}\beta(s)^kds - 2h\int_h^{\infty}\beta(s)ds - 2\int_0^hs\beta(s)^kds
\end{equation}
and we only need to show that the first term is the dominant as \(h\) grows. For the second term we have the asymptotics 
\begin{align}
2h\int_h^{\infty}\beta(s)^k ds\sim -\f{2h^{\alpha k+2}}{\alpha k+1}
\end{align}
and due to the integrability condition here \(\alpha k+2 <1\). Similarly for the third term 
\begin{align}
2\int_0^h s\beta(s)^k \sim \f{2h^{\alpha k+2}}{\alpha k+2}
\end{align}
which concludes this case. \\
\textbf{Case 3:} we use again \eqref{eq: change of variables} and now the first term behaves as
\begin{equation}
\int_0^h\beta(s)^k ds \sim \int_0^1 c ds + \int_1^h Ks^{-1} ds \sim K\log(h) + c 
\end{equation}
whereas the second one 
\begin{equation}
\int_0^h s\beta(s)^k ds \sim \int_0^1 cs ds + \int_1^h Ks^{0} ds \sim Kh + c
\end{equation}
putting it all together in \eqref{eq: change of variables} we get 
\begin{equation}
\int_{[a,a+h]^2}\beta(s-t)^k dsdt \sim 2Kh\log(h) 
\end{equation}
\end{proof}

\begin{lemma}
    \label{lem: sum asymptotics}
The following asymptotics hold as \(N\) grows
\begin{enumerate}
    \item If \(\alpha\in(-1,0)\): \begin{equation}
        \f{1}{N}\sum_{k=1}^{N-1}(1-k/N)(k/N)^{\alpha} \sim \f{1}{\alpha+1} - \f{1}{\alpha+2}
    \end{equation}
    \item If \(\alpha = -1\): \begin{equation}
        \f{1}{N}\sum_{k=1}^{N-1}(1-k/N)(k/N)^{-1} = \log(N) + \gamma - 1 + \frac{1}{2N} + o(N^{-2})
    \end{equation}
    \item If \(\alpha<-1\): \begin{equation}
        \f{1}{N}\sum_{k=1}^{N-1}(1-k/N)(k/N)^{\alpha} \sim \zeta(-\alpha)(1/N)^{\alpha+1}
    \end{equation}
    where \(\zeta(\cdot)\) is the Riemann zeta function.
\end{enumerate}
\end{lemma}
\begin{proof}

\textbf{Case 1:} in this case the summability of \(f(x)=(1-x)x^{\alpha}\) on \((0,1)\) leads to a Riemann sum type of approximation argument
\begin{equation}
\lim_{N\to\infty}\f{1}{N}\sum_{k=1}^{N-1}(1-k/N)(k/N)^{\alpha} = \int_0^1f(x)dx = \f{1}{\alpha+1} - \f{1}{\alpha+2}
\end{equation}

\textbf{Case 2:} in this case we use the well-known asymptotics of the harmonic numbers \(H_N\) which are defined as \(H_N = \sum_{k=1}^N 1/k\) and satisfy the following asymptotics as \(N\to\infty\):
\begin{equation}
H_N = \log(N) + \gamma + \frac{1}{2N} + o(N^{-2})
\end{equation}
where \(\gamma\) is the Euler-Mascheroni constant. Using this we can write 
\begin{align}
\f{1}{N}\sum_{k=1}^{N-1}(1-k/N)(k/N)^{-1} &= \sum_{k=1}^{N-1} \frac{1}{k} - \f{1}{N}\sum_{k=1}^{N-1} 1\\
&= H_{N-1} - 1 = \log(N) + \gamma - 1 + \frac{1}{2N} + o(N^{-2})
\end{align}
\textbf{Case 3:} Let \(\alpha\neq -1\). To proceed with this case we need to introduce the Hurwitz zeta function 
\[
\zeta(-\alpha, a) = \sum_{n=1}^{\infty}(n+a)^{\alpha}
\]
where \(\zeta(-\alpha,1) = \zeta(-\alpha)\) is just the Riemann zeta function. This function is useful as it allows to write 
\[
\sum_{k=1}^{N-1}k^{\alpha} = \zeta(-\alpha)-\zeta(-\alpha, N)
\]
and it satisfies the following asymptotics for large \(s\),  and \(|\arg(s)|<\pi\) (see \cite{magnusFormulasTheoremsSpecial1966} p.25 where a first order approximation is
enough for our purposes, we also disregard the constants as they will not appear in the leading term of our quantity of interest). 
\[
\zeta(-z, s)\sim \frac{1}{2}s^{z} + \frac{1}{-(z+1)}s^{1+z} + o(s^{z-1}).
\]
Let \(S_N =\f{1}{N}\sum_{k=1}^{N-1}(1-k/N)(k/N)^{\alpha}\) then we can write 
\[
S_N = (1/N)^{\alpha+1}(\zeta(-\alpha) - \zeta(-\alpha, N)) - (1/N)^{\alpha+2}(\zeta(-(\alpha+1)) - \zeta(-(\alpha +1),N))
\]
The goal now is to use the asymptotics above to show that \(\zeta(-\alpha,N)\xrightarrow{N\rightarrow\infty}0\) and the leading term is \(\zeta(-\alpha)(1/N)^{\alpha+1}\). 
Note that 
\begin{align}
N^{-(\alpha+1)}\zeta(-\alpha, N)&\sim \frac{1}{2}N^{-1} + \frac{1}{-(\alpha+1)} + o(N^{-2})\\
N^{-(\alpha+2)}\zeta(-(\alpha+1),N)&\sim \frac{1}{2}N^{-1}+\frac{1}{-(\alpha+2)} + o(N^{-2})
\end{align} 
Therefore 
\[
S_N\sim \zeta(-\alpha)(1/N)^{\alpha+1} - \zeta(-(\alpha+1))(1/N)^{\alpha+2} + \frac{1}{(\alpha+1)(\alpha+2)} + o(N^{-2})
\]
which concludes the result.

In the case \(\alpha = -2\) we need to be slightly more careful as for the second term zeta functions are not defined but the dominant term remains unchanged. Note that 
\begin{align}
(1/N)\sum_{k=1}^{N-1}(1-k/N)(k/N)^{-2} & = N\sum_{k=1}^{N-1}k^{-2} - \sum_{k=1}^{N-1}k^{-1}\\
&\sim N\zeta(2) - H_{N-1}\\
&\sim N\zeta(2) - \log(N) - \gamma + 1 + \frac{1}{2N} + o(N^{-2})
\end{align}
and therefore the leading term is still \(\zeta(2)N\) as claimed.
\end{proof}
\begin{lemma}
\label{lem: triple integral asymptotics}
Let \(\nu(s)\) be such that \(|\nu(s)|\leq C\) for all \(s\in\mathbb{R}\) and \(|\nu(s)|\sim K|s|^{\gamma}\) as \(s\to\infty\) for some \(\gamma<0\). Let \(\alpha, \beta\) and
\(m\) be such that \(\alpha+\beta+ (m-1)\gamma >-2\). Then as \(h\to\infty\) we have that
\begin{equation}
\int_{[0,h]^3}(s-u)_+^{\alpha}(s-v)_+^{\beta}\nu(u-v)^{m-1}dsdudv \asymp h^{\alpha+\beta+(m-1)\gamma+2}
\end{equation}
\end{lemma}
\begin{proof}
\begin{align}
\int_{[0,h]^3}(s-u)_+^{\alpha}(s-v)_+^{\beta}\nu(u-v)^{m-1}dsdudv &= \begin{multlined}
\int_0^h\int_0^s\int_0^s(s-u)^{\alpha}(s-v)^{\beta} \\ \nu(u-v)^{m-1}dudvds
\end{multlined}\\
&=\int_0^h\int_{[0,s]^2}x^{\alpha}y^{\beta}\nu(|x-y|)^{m-1}dxdyds\\
&=\int_0^h F(s)ds
\end{align}
and we study now the behaviour of \(F(s)\) in the different regimes. Around the diagonal we have that
\begin{align}
\int_{\{(x,y)\in[0,s]^2: |x-y|\leq 1\}}x^{\alpha}y^{\beta}&\nu(|x-y|)^{m-1}dxdy  \\
&\asymp\int_0^s\int_{(x-1)\vee 0}^{x+1} x^{\alpha}y^{\beta}dxdy\\
&\asymp\int_0^s x^{\alpha} \left((x+1)^{\beta+1} - ((x-1)_+)^{\beta+1}\right)dx\\ 
&=\int_0^1 x^{\alpha}(x+1)^{\beta+1}dx + \int_1^s x^{\alpha}\left((x+1)^{\beta+1} - (x-1)^{\beta+1}\right)dx\\
&\asymp \begin{cases}
    s^{\alpha+\beta+1}\quad &\alpha + \beta>-1, s>1\\
    \log(s) \quad &\alpha + \beta = -1, s>1\\
    1 \quad &s\leq 1
\end{cases}
\end{align}
and away from the diagonal we have that 
\begin{align}
\int_{\{(x,y)\in[0,s]^2: |x-y|> 1\}}x^{\alpha}y^{\beta}&\nu(|x-y|)^{m-1}dxdy\\
&= \int_{\{(x,y)\in[0,s]^2: |x-y|> 1\}}x^{\alpha}y^{\beta}|x-y|^{\gamma}dxdy\\
&= s^{\alpha+\beta+\gamma+2}\int_{\{(x,y)\in[0,1]^2: |x-y|> 1/s\}}x^{\alpha}y^{\beta}|x-y|^{\gamma}dxdy\\
&\asymp s^{\alpha+\beta+\gamma+2}\int_{[0,1]^2}x^{\alpha}y^{\beta}|x-y|^{\gamma}dxdy
\end{align}
provided that the latter integral is finite, which we now show is the case if \(\alpha+\beta+\gamma > -2\). 
\begin{align}
\int_{[0,1]^2}x^{\alpha}y^{\beta}|x-y|^{\gamma}dxdy & = \int_0^1\int_0^x x^{\alpha}y^{\beta}(x-y)^{\gamma}dydx + \int_0^1\int_x^1 x^{\alpha}y^{\beta}(y-x)^{\gamma}dydx\\
&= \int_0^1\int_0^x x^{\alpha}y^{\beta}(x-y)^{\gamma}dydx + \int_0^1\int_0^y x^{\alpha}y^{\beta}(y-x)^{\gamma}dydx
\end{align}
For the first term we can do 
\begin{align}
\int_0^1 x^{\alpha}\int_0^x y^{\beta}(x-y)^{\gamma}dydx & = \int_0^1 x^{\alpha+1} \int_0^1 (xt)^{\beta}(x-xt)^{\gamma}dt dx\\
&= \int_0^1 x^{\alpha+\beta+\gamma+1}\int_0^1 t^{\beta}(1-t)^{\gamma}dtdx\\
&= \beta(\beta+1, \gamma+1)\int_0^1 x^{\alpha+\beta+\gamma+1}dx\\
&= \frac{\beta(\beta+1, \gamma+1)}{\alpha+\beta+\gamma+2}
\end{align}
and for the second one we can get through the exact same argument that it is equal to \(\frac{\beta(\alpha+1, \gamma+1)}{\alpha+\beta+\gamma+2}\). We can thus conclude that 
the behaviour away from the diagonal dominates since 
\begin{equation}
F(s)\asymp s^{\alpha+\beta+\gamma+2}  + G(s) + 1
\end{equation}
where
\begin{equation}
G(s)\asymp\begin{cases}
    s^{\alpha+\beta+1}\quad &\alpha + \beta>-1, s>1\\
    \log(s) \quad &\alpha + \beta = -1, s>1
\end{cases}
\end{equation}
and thus in any case it is dominated by the first term as \(s\to\infty\) since \(\gamma > -1\). We thus conclude the proof by going back to the original integral and writing 
\begin{equation}
\int_0^h F(s)ds \asymp \int_0^1 1ds + \int_1^h s^{\alpha+\beta+\gamma+2}ds \asymp h^{\alpha+\beta+\gamma+3}
\end{equation}
\end{proof}
\section{Auxiliary Integrals}
For the sake of completeness we include here some auxiliary integrals that are used in the main proofs. These are all based on the properties of the Beta function and known results. 
\begin{lemma}
    \label{lem: aux integral 1}
For any \(1/2<H<1\) we have that 
\begin{equation}
\label{eq: will see}
\int_{-\infty}^{u\vee v}(u-x)^{H-3/2}(v-x)^{H-3/2}dx = \beta(H-1/2,2-2H)|u-v|^{2H-2}
\end{equation} 
\end{lemma}
\begin{proof}
\begin{align}
\int_{-\infty}^{u\vee v}(u-x)^{H-3/2}&(v-x)^{H-3/2}dx= \int_0^{\infty}(s+|u-v|)^{H-3/2}s^{H-3/2}ds\\
 &=|u-v|^{2H-3}\int_0^{\infty}\left(\frac{s}{|u-v|}\right)^{H-3/2}\left(\frac{s}{|u-v|}\right)^{H-3/2}ds\\ 
 &=|u-v|^{2H-2}\int_0^{\infty}(1+t)^{H-3/2}t^{H-3/2}dt\\
 &=\beta(H-1/2,2-2H)|u-v|^{2H-2}
\end{align}
\end{proof}
\begin{lemma}
    \label{lem: aux integral 2}
For any \(1/2<H<1\) we have that
\begin{equation}
\int_{u\vee v}^{1}(s-u)^{H-3/2}(s-v)^{H-3/2}ds < |u-v|^{2H-2}\beta(H-1/2,2-2H).
\end{equation}
\begin{proof}
\begin{align}
\int_{u\vee v}^{1}(s-u)^{H-3/2}&(s-v)^{H-3/2}ds = \int_0^{1-u\vee v}(s+|u-v|)^{H-3/2}s^{H-3/2}ds\\
&=|u-v|^{2H-3} \int_0^{1-u\vee v}\left(\frac{s}{|u-v|}\right)^{H-3/2}\left(\frac{s}{|u-v|}+1\right)^{H-3/2}ds\\
&=|u-v|^{2H-2}\int_0^{(1-u\vee v)/|u-v|}t^{H-3/2}(t+1)^{H-3/2}dt\\
&< |u-v|^{2H-2}\int_0^{\infty}t^{H-3/2}(t+1)^{H-3/2}dt\\
&=|u-v|^{2H-2}\beta(H-1/2,2-2H)
\end{align} 
\end{proof}
\end{lemma}
\begin{lemma}
    \label{lem: aux integral 3}
For any \(3/4<H<5/6\) we have that, if \(u>v\)  
\begin{equation}
    \int_u^1 (s-u)^{H-3/2}(s-v)^{2H-2}ds < |u-v|^{3H-5/2}\beta(H-1/2, 5/2-3H)
\end{equation}
whereas if \(v>u\) we have that 
\begin{equation}
    \int_v^1 (s-u)^{H-3/2}(s-v)^{2H-2}ds < |u-v|^{3H-5/2}\beta(2H-1, 5/2-3H) 
\end{equation}
\end{lemma}
\begin{proof}
In the first case we have 
\begin{align}
 \int_u^1 (s-u)^{H-3/2}(s-v)^{2H-2}ds &= (u-v)^{3H-5/2}\int_0^{(1-u)/(u-v)}s^{H-3/2}(1+s)^{2H-2}ds\\
 & < (u-v)^{3H-5/2}\int_0^{\infty}s^{H-3/2}(1+s)^{2H-2}ds 
\end{align} 
and in the second case we have 
\begin{align} 
\int_v^1 (s-u)^{H-3/2}(s-v)^{2H-2}ds &= (v-u)^{3H-5/2}\int_0^{(1-v)/(v-u)}s^{2H-2}(1+s)^{H-3/2}ds \\
&< (v-u)^{3H-5/2}\int_0^{\infty}s^{2H-2}(1+s)^{H-3/2}ds
\end{align}
\end{proof}
\begin{lemma}
    \label{lem: aux integral 4}
For any \(3/4<H<1\) and \(u>v\) we have that 
\begin{equation}
\int_v^u (s-u)^{H-3/2}(s-v)^{2H-2}ds = (u-v)^{3H-5/2}\beta(H-1/2, 2H-1)
\end{equation}
\end{lemma}
\begin{proof}
\begin{align}
\int_v^u (s-u)^{H-3/2}(s-v)^{2H-2}ds &= \int_0^{u-v}s^{H-3/2}(s+u-v)^{2H-2}ds\\
&= (u-v)^{3H-5/2}\int_0^1s^{H-3/2}(s+1)^{2H-2}ds\\
&= (u-v)^{3H-5/2}\beta(H-1/2, 2H-1)
\end{align}
\end{proof}
\section{The fOU kernel}
We note that the fastened fractional OU process we are interested in admits the following kernel representation when \(H>1/2\):
\begin{equation}
Y^{\varepsilon}_t = p\,\varepsilon^{-1/2}\int_{\mathbb{R}}e^{-\frac{t-s}{\varepsilon}}\int_0^{\frac{t-s}{\varepsilon}}e^v (v)_+^{H-3/2}dv dW_s
\end{equation}
let for now \(g(s)=e^{-s}\int_0^se^v v^{H-3/2}dv\) so that the kernel representation can be written as
\begin{equation}
Y^{\varepsilon}_t = p\,\varepsilon^{-1/2}\int_{-\infty}^t g\left(\frac{t-s}{\varepsilon}\right)dW_s
\end{equation}
we study now the asymptotics of this function \(g(s)\). The function \(g\) is continuous and bounded on \(\mathbb{R}\) and it satisfies \(g(0)=0\) and 
\(g(s)\sim s^{H-1/2}\) as \(s\to 0\) and \(g(s)\sim s^{H-3/2}\) as \(s\to\infty\). The asymptotic behaviour near \(0\) is pretty much irrelevant but the 
important bit is that there exists a finite value \(s^*\) such that \(g(s^*) = \sup_{s\in\mathbb{R}}g(s)\) and before that value we can bound the kernel by 
a constant and after that value we can use its asymptotic behaviour. Note that 
\[\lim_{s\to\infty}g(s)s^{3/2-H} = 1\]
in the sense that there is no constant missing there. Let now \(\tau_t K^{\varepsilon}(x) = \varepsilon^{-1/2}g\left(\frac{t-x}{\varepsilon}\right)\). We are interested 
in the quantity: 
\begin{equation}
\int_{-\infty}^{t\wedge s}\tau_tK^{\varepsilon}(x)(s-x)^{H-3/2}dx
\end{equation}
In the case when \(s<t\) then \(t-x>t-s\) and therefore we can use the asymptotic behaviour of the kernel to get that, as \(\varepsilon\to 0\):
\begin{equation}
\int_{-\infty}^{t\wedge s}\tau_tK^{\varepsilon}(x)(s-x)^{H-3/2}dx\sim \varepsilon^{1-H}(t-s)^{2H-2}\beta(H-1/2,2-2H)
\end{equation}
The case when \(s>t\) is more delicate as we need to filter out the constant behaviour. We can write 
\begin{align}
\int_{-\infty}^{t}\tau_tK^{\varepsilon}(x)(s-x)^{H-3/2}dx & =\int_{-\infty}^{t-s^*\varepsilon}\tau_tK^{\varepsilon}(x)(s-x)^{H-3/2}dx + \\ &\int_{t-s^*\varepsilon}^t\tau_tK^{\varepsilon}(x)(s-x)^{H-3/2}dx
\\
&=A_t^{<} + A_t^{\varepsilon}
\end{align}
In the first term we can use the asymptotic behaviour of the kernel to get that, as \(\varepsilon\to 0\):
\begin{align}
A_t^{<} & \sim \varepsilon^{1-H}\int_{-\infty}^{t-s^*\varepsilon}(t-x)^{H-3/2}(s-x)^{H-3/2}dx\\
&=\varepsilon^{1-H}\int_{-\infty}^{-s^*\varepsilon}(-x)^{H-3/2}(-x + (s-t))^{H-3/2}dx\\
&=\varepsilon^{1-H}(s-t)^{2H-3}\int_{s^*\varepsilon}^{\infty}\left(\frac{x}{s-t}\right)^{H-3/2}\left(\frac{x}{s-t} + 1\right)^{H-3/2}dx\\
&=\varepsilon^{1-H}(s-t)^{2H-2}\int_{s^*\varepsilon/(s-t)}^{\infty}t^{H-3/2}(t+1)^{H-3/2}dt\\
&\sim \varepsilon^{1-H}(s-t)^{2H-2}\beta(H-1/2,2-2H)
\end{align}
Instead the second term is controlled by 
\begin{align}
A_t^{\varepsilon} &\lesssim \varepsilon^{-1/2}\int_{t-s^*\varepsilon}^t(s-x)^{H-3/2}dx\\
&\leq \varepsilon^{-1/2}\int_{t-s^*\varepsilon}^t(s-t)^{H-3/2}dx =\varepsilon^{1/2}(s-t)^{H-3/2} 
\end{align}
where we simply use the boundedness of the kernel as that dominates the behaviour there. We sum up this behaviour in the following lemma. 
\begin{lemma}
    \label{lem: fOU kernel asymptotics}
Let \(\tau_t K^{\varepsilon}(x) = \varepsilon^{-1/2}g\left(\frac{t-x}{\varepsilon}\right)\) where \(g(s) = e^{-s}\int_0^se^vv^{H-3/2}dv\) be the kernel of the fastened fractional OU process for \(H>1/2\). Then as \(\varepsilon\to 0\) we have 
that \begin{itemize}
    \item If \(s<t\) then as \(\varepsilon\to 0\):
    \begin{equation}
    \int_{-\infty}^{t\wedge s}\tau_tK^{\varepsilon}(x)(s-x)^{H-3/2}dx \sim \varepsilon^{1-H}(t-s)^{2H-2}\beta(H-1/2,2-2H)
    \end{equation}
    \item If \(s>t\) then as \(\varepsilon\to 0\) we have the following decomposition \begin{equation}
\int_{-\infty}^{t\wedge s}\tau_tK^{\varepsilon}(x)(s-x)^{H-3/2}dx = A_t^{<} + A_t^{\varepsilon}
    \end{equation}
    where the first term behaves as \begin{equation}
A_t^{<} \sim \varepsilon^{1-H}(s-t)^{2H-2}\beta(H-1/2,2-2H)
    \end{equation}
and the second term is controlled by
\begin{equation}
A_t^{\varepsilon} \lesssim \varepsilon^{1/2}(s-t)^{H-3/2} 
\end{equation}
\end{itemize}

\end{lemma}
\endappendix

\newpage
\pagestyle{empty}
\bibliographystyle{plain}
\bibliography{bibliography}
\end{document}